\documentclass[final]{siamltex}
\usepackage{amsmath}
\usepackage{mathtools}
\usepackage{amssymb}
\usepackage{amsfonts}
\usepackage{amsxtra}
\usepackage{amstext}
\usepackage{amsbsy}
\usepackage{amscd}
\usepackage{graphicx}
\usepackage{float}
\usepackage{cite}
\usepackage{lmodern}
\usepackage{xcolor}
\usepackage{srcltx} 
\usepackage{marginnote,slashbox}
\usepackage{rotating}
\definecolor{darkblue}{rgb}{0.0,0.0,0.6}
\definecolor{darkgreen}{rgb}{0.0,0.6,0.0}
\usepackage[pdftex,colorlinks=true,urlcolor=darkblue,citecolor=darkblue,linkcolor=darkblue]{hyperref}
\usepackage[utf8]{inputenc}      
\usepackage{booktabs}            
\usepackage{url}                 
\usepackage[T1]{fontenc}         
\usepackage{algorithmic}         

\usepackage{calc}
\usepackage[notcite,notref]{showkeys}

\numberwithin{table}{section}    
\numberwithin{figure}{section}   
\numberwithin{equation}{section} 

\usepackage{my_latex_commands}

\newtheorem{assumption}[theorem]{Assumption}
\newtheorem{remark}[theorem]{Remark}

\renewcommand{\bl}[1]{\textcolor{black}{#1}}

\newcommand{\maxlim}[1]{\max\nolimits{#1}}
\newcommand{\dy}{\delta y}
\newcommand{\du}{\delta u}
\newcommand{\dl}{\delta \ell}

\newcommand{\df}{\delta \varphi}
\newcommand{\dd}{\delta d}
\newcommand{\yy}{\bar y}
\newcommand{\uu}{\bar u}
\newcommand{\lu}{L^2(0,T;U)}
\renewcommand{\ll}{\bar \ell}
\newcommand{\lU}{L^2(0,T;U^\ast)}
\newcommand{\lH}{L^2(0,T;Z)}
\newcommand{\lV}{L^2(0,T;V)}
\newcommand{\lY}{L^2(0,T;Y^\ast)}
\newcommand{\ly}{L^2(0,T;Y)}
\newcommand{\lo}{L^2(\O)}
\newcommand{\hon}{H^1_0(\O)}
\newcommand{\hoon}{H^{-1}(\O)}
\newcommand{\ho}{H^1(\O)}
\newcommand{\hoo}{H^1(\O)^\ast}
\newcommand{\hy}{H^1_0(0,T;Y)}
\newcommand{\hY}{H^1_T(0,T;Y^\ast)}
\newcommand{\lly}{L^2(0,T;Y)}
\newcommand{\llx}{L^2(0,T;X)}
\newcommand{\dense}{\overset{d}{\embed}}
\renewcommand{\o}{\omega}
\newcommand{\V}{\mathcal{V}}
\newcommand{\doo}{\delta \o}
\renewcommand{\bo}{\bar \omega}

\newcommand{\hhyy}{H^1_0(0,T;H^1_0(\O))}
\newcommand{\llU}{L^2(0,T;\hoo)}
\newcommand{\llu}{L^2(0,T;H^{1}(\Omega))}
\newcommand{\llun}{L^2(0,T;H^{-1}(\Omega))}
\newcommand{\llz}{L^2(0,T; L^2(\Omega))}

\newcommand{\hhyn}{H^1_0(0,T;\lo)}
\newcommand{\F}{(\kappa \circ H)}

\begin{document}

\title{Strong Stationarity for the  control of viscous history-dependent Evolutionary VIs arising in Applications
}
\date{\today}
\author{Livia\ Betz\footnotemark[1]}
\renewcommand{\thefootnote}{\fnsymbol{footnote}}
\footnotetext[1]{Faculty of Mathematics, University of W\"urzburg,  Germany.}
\renewcommand{\thefootnote}{\arabic{footnote}}

\maketitle

 \normalsize
 
 \begin{abstract}
 This paper addresses  optimal control problems governed by history-dependent EVIs with viscosity. One of the prominent properties of the state system is its non-smooth nature, so that the application of standard adjoint calculus is excluded. We extend the results from \cite{st_coup} by showing that history-dependent EVIs with viscosity can be formulated as non-smooth ODEs in Hilbert space in a general setting. The Hadamard directional differentiability of the solution map is investigated. Based on previous results, this allows us to 
establish strong stationary conditions for two different viscous damage models with fatigue.
 \end{abstract}


\section{Introduction}\label{sec:i}

The study of history-dependent evolutionary variational inequalities (EVIs) has attracted much attention lately, see e.g.\,\cite{mig2021, mig_bai} (frictional contact), \cite{mig_sof} (sweeping processes) and the references therein. Most of the papers on this topic address the existence and the regularity  of solutions. For a comprehensive study of history-dependent EVIs \textit{with viscosity}  we refer to \cite[Chp.\,4.4]{sofonea} where the applications focus on total slip-dependent frictional contact problems involving viscoelastic materials \cite[Chp.\,10.4]{sofonea}.

When it comes to the \textit{optimal control} of history-dependent  variational inequalities,  there are only a few papers available, most of which are concerned with the existence of optimal solutions, cf.\,for instance \cite{oc_h_d, oc_sof}. 
Because of   the non-smooth nature of EVIs, the derivation of optimality conditions is a challenging issue. Indeed, the literature on differentiability properties of EVIs is rather scarce, see \cite{conical, cc, qevi} (EVIs of obstacle type) and \cite{susu_dam, st_coup, hal} (viscous EVIs). 
We point out that these contributions do not take a history operator into account, except \cite{hal}, where a concrete application is considered.
To the best of our knowledge,  the sensitivity analysis of viscous history-dependent EVIs in a general framework has not been examined so far, let alone  the \textit{strong stationarity} for the control thereof.

This paper aims at addressing this particular aspect. We establish strong stationary optimality conditions for the control of two damage models with fatigue. Both these models fall into the category of viscous history-dependent EVIs. In a general framework, this type of evolution is described as follows
\begin{equation}\label{evi}\tag{EVI}
 R(\HH(y)(t),\eta)-R(\HH(y)(t),\dot y(t))+\dual{\V\dot y(t)}{\eta-\dot y(t)}_Y  \geq \dual{g(y(t),\ell(t))}{\eta-\dot y(t)}_Y \quad \forall\, \eta \in Y, 
\end{equation}\text{\ae}$(0,T)$, where $Y$ is a Hilbert space, $\HH$ is the history operator and $\VV$ denotes the viscosity. In the present paper, the dissipation $R$ may take infinite values and $g$ is a directionally differentiable mapping. The precise assumptions on the data are stated in Assumption \ref{assu:st} below. 
The existence of solutions in a slightly less general framework has already been addressed in \cite[Chp.\,4.4]{sofonea}. In this paper we go one step further, by formulating \eqref{evi} as a non-smooth ODE in Hilbert space. This facilitates the   investigation of the directional differentiability of the solution map associated to \eqref{evi}. By resorting to previous findings \cite{st_coup}, we are then able  to establish strong stationarity for the control of the two applications mentioned above.

From the point of view of optimal control, the essential feature of the problem under consideration is that it has a \textit{non-smooth} character, so that the standard methods for the derivation of qualified optimality conditions are not applicable here. 
The key novelties of the present paper are:
\begin{itemize}
\item the equivalence of \eqref{evi} to a {concrete non-smooth ODE} in Hilbert space; in this context we give an explicit formula for the underlying non-smooth non-linearity 
\item sensitivity analysis for viscous \textit{history-dependent}  evolutionary VIs in a general framework 
\item optimal control for damage models with fatigue  in terms of {strong stationarity }
\end{itemize}

Deriving necessary optimality conditions is a challenging issue even in finite dimensions, where a special attention is given to MPCCs  (mathematical programs with complementarity constraints).
In \cite{ScheelScholtes2000} a detailed overview of various optimality conditions of different strength was introduced, see also \cite{HintermuellerKopacka2008:1} for the infinite-dimensional case. The  most rigorous stationarity concept is strong stationarity.
Roughly speaking, the strong stationarity conditions involve an optimality system, which is equivalent to the purely primal conditions saying that the directional derivative of the reduced objective in feasible directions is nonnegative (which is referred to as B stationarity).

When it comes to establishing optimality conditions for the control of EVIs (and of non-smooth processes in general), most of the authors resort to smoothening procedures. {This approach is meanwhile standard  \cite{Barbu:1981:NCD} when dealing with the control of non-smooth evolutions see e.g.\,\cite{Barbu:1981:NCD, barbu84, tiba, ItoKunisch2008:1, tiba_hyp, barbu_tiba} and the references therein.
The optimality systems derived in this way are of intermediate strength and are not expected to be of strong stationary type, since one always loses information when passing to the limit in the regularization scheme see e.g.\,\cite[Rem. 3.11]{mcrf} and \cite[Subsec 7.2]{paper}. Thus, proving strong stationarity for the optimal control of  non-smooth  problems requires direct approaches, which employ the limited differentiability properties of the control-to-state map. 
 Based on the pioneering work  \cite{mp76} (strong stationarity  for  optimal control of elliptic VIs of obstacle type), most of them focus on elliptic VIs, see e.g.\,\cite{mp84, e_qvi, DelosReyes-Meyer} and the references therein. Regarding strong stationarity for  optimal control of non-smooth evolution processes, the literature is very scarce and the only  papers known to the author addressing this issue so far are \cite{brok_ch, qevi, st_coup, cc} (EVIs) and \cite{paper, st_coup, frac} (time-dependent PDEs/ODEs). We point out that, in contrast to our problem, all the above mentioned contributions on the topic of strong stationarity for EVIs do not take a history operator into account.

Let us give an overview of the main contributions in this paper. After introducing the notation, we recall in section \ref{sec:pres} a result from \cite{st_coup} concerning strong stationarity  for the optimal control of non-smooth coupled systems. 
Together  with our main finding from section \ref{sec:evi}, this will  allow us  to establish strong stationarity for  two concrete applications (section \ref{sec:dm}).

In section \ref{sec:evi}, we show that  viscous history-dependent EVIs are non-smooth ODEs in Hilbert space (Theorem \ref{thm:ode}). Here we extend the results from the previous work \cite{st_coup}, where such a characterization was introduced for viscous EVIs which do not involve history.
We give an explicit formula for the non-smooth non-linearity in the ODE in terms of a projection operator (Definitions \ref{proj}, \ref{def:fp}). By contrast to \cite{st_coup}, the non-smoothness in the present paper has two arguments instead of one.  Its concrete description provides multiple advantages. Firstly, it simplifies the  solvability theory and allows us to easily examine the existence and the regularity of solutions for discontinuous right-hand sides by means of a classical fix point argument (cf.\,also Remark \ref{p}). Secondly, the characterization of \eqref{evi} in terms of a non-smooth ODE  plays an essential role in the context of sensitivity analysis and optimal control. In fact, with the explicit formulation of the non-smoothness at hand, we can state conditions on the  projection operator such that the  Hadamard \textit{directional differentiability} of the solution map to \eqref{evi} is guaranteed (Theorem \ref{thm:ode_dir}).

Section  \ref{sec:dm} focuses on proving strong stationarity for the optimal control of two viscous  gradient damage models with \textit{fatigue}. Here we are concerned with the application of the above mentioned  results. We first employ the main findings from section \ref{sec:evi} to show that these concrete applications can be rewritten as non-smooth ODEs, after which we make use of the result from section \ref{sec:pres} to derive strong stationary optimality conditions.
In subsection \ref{sec:0}, the viscosity is expressed in terms of the $H^1_0$ norm and it describes the evolution of a single damage variable. In subsection \ref{sec:1}, a penalization approach is employed such that the model becomes two-field. This allows us to work with $L^2$ viscosity. It  has the advantage that the  non-smoothness appearing in the ODE is then expressed by means of the Nemytskii operator associated to $\max(\cdot,0)$.
This is not the case in the single-field model, where the underlying non-smooth function features the projection onto a subset of $\hoon$. As we will see, the more accurate description of the damage evolution in the two-field model carries over to the associated directional differentiability and the strong stationarity conditions.


\subsection*{Notation}
Throughout the paper, $T > 0$ is a fixed final time. If $X$ and $Y$ are linear normed spaces, then the space of linear and bounded operators from 
$X$ to $Y$ is denoted by $\LL(X,Y)$, and $X \overset{d}{\embed} Y$ means that $X$ is \bl{densely embedded} in $Y$.
The dual space of  $X$ will be denoted by $X^*$, except the dual of the space ${H_0^1(\O)}$ which is denoted by $\hoon$. For the dual pairing between $X$ and $X^*$
we write $\dual{.}{.}_X$.  The \bl{closed} ball in $X$ around $x \in X$ with radius $\alpha >0$ is denoted by $B_X(x,\alpha)$. \bl{If $X$ is a Hilbert space, we write $(\cdot,\cdot)_X$ for the associated scalar product.}
The  following abbreviations will be used throughout the paper:
\begin{equation*}
\begin{aligned}
H^{1}_0(0,T;X)&:=\{z\in H^{1}(0,T;X):z(0)=0\},
\\H^{1}_T(0,T;X)&:=\{z\in H^{1}(0,T;X):z(T)=0\},
\end{aligned}
\end{equation*}where $X$ is a Banach space.
For the polar cone of a set $ M \subset X$ we use the notation $M^\circ:=\{x^\ast \in X^\ast : \dual{x^\ast}{x}_X \leq 0 \quad \forall\, x \in M\}.$ By  $\raisebox{3pt}{$\chi$}_{M}$ we denote the characteristic function associated to the set $M$. Given $x \in X^\ast$, we denote its annihilator by $ [x]^{\perp}:=\{\mu \in X: \dual{x}{ \mu}_ {X}=0 \} $. 
Derivatives w.r.t.\ time (weak derivatives of vector-valued functions) are frequently denoted by a dot.
The symbol $\partial f$ stands for the convex subdifferential and by $\dom(f)$ we denote the domain of the  functional $f:X \to (-\infty,\infty]$, see e.g.\ \cite{rockaf}.
For a mapping $\RR:X \times Y \to (-\infty,\infty]$, 
the set $\partial_{2}\RR(x,y) \subset Y^\ast$ describes the convex subdifferential of the functional $\RR(x,\cdot):Y \to (-\infty,\infty]$ in $y$. 
The Nemystkii-operators associated with the mappings considered in this paper will be described by the same symbol, even when considered with different domains and ranges. By $\max(\cdot, 0)$ we denote the positive part function, while $\max'(x; h)$ indicates its directional derivative  in the point $x$ in direction $h$. Similarly,  $\min(\cdot, 0)$ stands for the negative part function.
With a little abuse of notation, we use in the paper the Laplace symbol for the operator $\Delta: \ho \to \hoo$ defined by 
\begin{equation*}
 \dual{\Delta \eta}{\psi}_{\ho} := - \int_\Omega \nabla \eta  \nabla \psi \,dx \quad \forall\,
 \psi \in \ho.
\end{equation*}The same symbol is used for the Laplace operator with domain $\hon$ and range $\hoon$.

 \section{A  strong stationarity result}\label{sec:pres}
In this section we recall a  known result \cite{st_coup} concerning the optimal control of non-smooth coupled systems (Theorem \ref{thm:strongstat} below). This states the strong stationarity optimality conditions for this particular type of state system and it will play an essential role later on in section \ref{sec:dm}.
 \begin{equation}\label{eq:optprob}
 \left.
 \begin{aligned}
  \min_{\ell \in L^2(0,T;V)} \quad & J(y,u,\ell)\\
  \text{s.t.} \quad & 
  \begin{aligned}[t]
   &\dot y(t) = f(\Phi(y,u)(t)) \quad \text{\ae } (0,T),\quad y(0) = 0,\\
     &\Psi(y,u)(t)=\ell(t)  \quad  \text{\ae } (0,T),\\
    & y \in H^1(0,T;Y), \quad u \in L^2(0,T;U).
  \end{aligned} 
 \end{aligned}
 \quad \right\}
\end{equation}
We begin by gathering all the necessary assumptions that are needed for the main result in Theorem \ref{thm:strongstat} to be true.
\begin{assumption}{\cite[Assumption 2.1]{st_coup}}\label{assu:stand}
For the quantities in \eqref{eq:optprob} we require the following:
 \begin{enumerate}
  \item\label{it:stand1}$V$, $Y$, and $U$ are  real reflexive Banach spaces, such that $V \dense U^\ast$.
   \item\label{it:stand3} The mappings $\Phi: L^2(0,T;Y \times U) \to L^2(0,T;Y^\ast)$ and $\Psi: L^2(0,T;Y \times U) \to L^2(0,T;U^\ast)$ are G\^ateaux-differentiable operators. 
    \item\label{it:stand2}
The non-smooth function  $f: Y^\ast \to Y$ is assumed to be Lipschitz continuous and directionally differentiable, i.e., 
  \begin{equation*}\label{eq:fdiff}
   \Big\|\frac{f(x + \tau \,h) - f(x)}{\tau} - f'(x;h)\Big\|_{Y} \stackrel{\tau \searrow 0}{\longrightarrow} 0 \quad \forall \, x,h \in Y^\ast.
  \end{equation*}
 \vspace{-0.3cm} \item\label{it:stand4} 
  The objective $J: L^2(0,T;Y)\times \lu \times L^2(0,T;V) \to \R$ 
  is Fr\'echet-differentiable. 
 \end{enumerate}
\end{assumption}
\begin{remark}\label{rem}
Note that in contrast to \cite[Assump.\,2.1.2]{st_coup}, we work  with operators $\Phi$ and $\Psi$ mapping between abstract function spaces, so that these operators are not necessary Nemytskii operators as in \cite[Sec.\,2]{st_coup}. This allows us to apply the findings in this section to applications which feature e.g.\,integral operators such as history operators, cf.\,Assumption \ref{assu:standd}.\ref{it:stand1d} below. A short inspection of \cite[Sec.\,2]{st_coup} shows that the entire analysis can be carried on in the same manner for our slightly more general setting without affecting the main result in Theorem \ref{thm:strongstat} below. 
\end{remark}

As in \cite[Sec.\,2]{st_coup}, the properties we need from the control-to-state map in order to prove the main result (Theorem \ref{thm:strongstat})  are just assumed to be true. To keep the {demonstration} concise, we do not discuss the unique solvability of the state system {nor} its differentiability properties.
These issues will be addressed in detail for the applications considered in section  \ref{sec:dm} below.
 \begin{assumption}{\cite[Assumption 2.3]{st_coup}}[Control-to-state operator]\label{assu:s_diff}
 \begin{enumerate} \item\label{it:s1}
 Throughout this section, we assume that  for every $\ell \in \lV$, the  state equation \begin{equation}\label{eq:pde}\left.
  \begin{aligned}
   &\dot y= f(\Phi(y,u)) \quad \text{\ae } (0,T),\ y(0) = 0,\\
     &\Psi(y,u)=\ell  \quad  \text{\ae } (0,T)
  \end{aligned} \quad \right\}
\end{equation}admits a unique solution $(y,u) \in H^1_0(0,T;Y) \times L^2(0,T;U) $ 
and denote the associated solution operator by 
$$\SS:L^2(0,T;V) \ni \ell \mapsto (y,u) \in H^1_0(0,T;Y) \times L^2(0,T;U).$$
 \item\label{it:s2} \bl{The mapping $\SS:L^2(0,T;V) \to \ly \times \lu $ is directionally differentiable}, i.e.,
  \begin{equation*}
   \Big\|\frac{\SS(\bl{\ell} + \tau \,\dl) - \SS(\bl{\ell})}{\tau} - \SS'(\bl{\ell};\dl)\Big\|_{\ly \times \lu } \stackrel{\tau \searrow 0}{\longrightarrow} 0 \quad \forall \ \bl{\ell,}\dl \in L^2(0,T;V).
  \end{equation*}
Moreover, we suppose that for any \bl{$\ell,\dl \in L^2(0,T;V)$}, the pair $(\delta y, \delta u):= \SS'(\bl{\ell};\dl) \in H^1_0(0,T;Y) \times L^2(0,T;U)$ is  the unique solution of \begin{subequations}\label{eq:pde_lin} \begin{gather}
   \dot \dy = f'\big(\Phi( y,u);\Phi'( y,u)(\dy,\du)\big) \  \text{\ae } (0,T),\ \delta y(0) = 0,\label{eq:pde1}\\
     \Psi'(  y,u)(\dy,\du)=\dl  \quad  \text{\ae } (0,T),\label{eq:pde2}
\end{gather}\end{subequations}where we abbreviate $\bl{( y, u)}:=\SS(\bl{\ell})$.

 \item\label{it:s3}  \bl{For any $\ell \in \lV$, there exists a  constant   $K>0$ so that}
 \begin{equation*}\label{eq:est}
\|\SS'(\bar \ell;\dl)\|_{\lly \times \lu} \leq K\, \|\dl\|_{L^2(0,T;U^\ast)} \quad \forall\, \dl \in \lV.\end{equation*}
\bl{If  $(\hat \dy, \hat \du) \in \hy \times \lu$ solves \eqref{eq:pde_lin} \bl{with} r.h.s.\ $\hat \dl \in \lU$ and if there exists a sequence $\{\dl_n\}_n \subset \lV$ with $\dl_n \to \hat \dl $ in $\lU$, then 
$\SS'(\bl{\ell};\dl_n) \to (\hat \dy, \hat \du)$ in $\lly \times \lu$.
}
 \end{enumerate}
\end{assumption}


Throughout this section one tacitly assumes that Assumptions \ref{assu:stand} and  \ref{assu:s_diff}  always hold true, without mentioning them everytime.

\begin{lemma}{\cite[Lemma 2.5]{st_coup}}[B-stationarity] If $\bar \ell \in L^2(0,T;V)$ is locally optimal for \eqref{eq:optprob}, then there holds
 \begin{equation}\label{eq:vi}
  \partial_{(y,u)} J(\SS(\bar \ell),\bar \ell )\SS'(\bar \ell;\dl)+\partial_{\ell} J(\SS(\bar \ell),\bar \ell ) \dl \geq 0   \quad \forall\, \dl\in L^2(0,T;V).
 \end{equation}
\end{lemma}
\begin{assumption}{\cite[Assumption 2.6]{st_coup}}\label{assu:surj}
For any local optimum $\bar \ell$ of \eqref{eq:optprob}, we  assume that \bl{$\rg(\partial_u \Phi(\bar y,\bar u))\dense \lY$,} where $(\bar y,\bar u):=\SS(\ll)$.
\end{assumption}

\begin{assumption}{\cite[Assumption 2.9]{st_coup}}\label{assu:reg}
For any local optimum $\bar \ell$ of \eqref{eq:optprob}, we assume that there exists $\lambda\in \ly$ so that $$ - \partial_u \Phi(\bar y,\bar u)^{\ast}  \lambda= \partial_u J(\bar y, \bar u, \ll)+\partial_u  \Psi (\bar y, \bar u)^*  \partial_\ell J(\bar y, \bar u, \ll) ,  $$where $(\bar y,\uu) : = \SS(\ll)$. \end{assumption}
 
 \begin{theorem}{\cite[Theorem 2.11]{st_coup}}[Strong Stationarity]\label{thm:strongstat} Suppose that Assumptions  \ref{assu:surj} and  \ref{assu:reg} are  satisfied. 
 Let $\bar \ell \in \lV$ be locally optimal for \eqref{eq:optprob} with associated state $(\bar y,\uu) : = \SS(\ll)$. 
 Then, there exist unique adjoint states
 \begin{equation*}
  \xi \in \hY \quad \text{and} \quad w \in \lu 
 \end{equation*}  
 and a unique multiplier $\lambda \in L^{2}(0,T;Y)$ such that the following system is satisfied 
 \begin{subequations}\label{eq:strongstat}
 \begin{gather}
  -\dot \xi- \partial_y  \Phi (\bar y, \bar u)^* \lambda+\partial_y  \Psi (\bar y, \bar u)^* w = \partial_y J(\bar y, \bar u, \ll) \  \text{ in  }\lY, \quad \xi(T) = 0, \label{eq:adjoint1}\\[1mm]
  - \partial_u  \Phi (\bar y, \bar u)^* \lambda +  \partial_u  \Psi (\bar y, \bar u)^* w= \partial_u J(\bar y, \bar u, \ll) \  \text{ in  }\lU, \label{eq:adjoint2}\\[1mm]
  \dual{\xi(t)}{f'(\Phi (\bar y,\bar u)(t) ; v)}_{Y}  \geq  \dual{\lambda(t)}{v}_{Y^\ast} 
  \quad \forall\, v\in Y^\ast, \;\text{a.e.\ in  }  \, (0,T), \label{eq:signcond} \\[1mm]
w+ \partial_{\ell} J(\bar y, \bar u, \ll)= 0 \  \text{ in  }\lu. \label{eq:grad}
 \end{gather} 
 \end{subequations}
\end{theorem}

\begin{theorem}{\cite[Theorem 2.13]{st_coup}}[Equivalence between B- and strong stationarity]\label{thm:equiv_B_strong} 
Assume that $\bar \ell \in L^2(0,T;V)$ together with its states $(\bar y, \bar u) \in \hy \times \lu$, some 
 adjoint states $(\xi,w) \in \hY \times \lu$, and a multiplier $\lambda \in \ly$ 
 satisfy the optimality system \eqref{eq:adjoint1}--\eqref{eq:grad}.
 Then, it also satisfies the variational inequality  \eqref{eq:vi}.  If \bl{Assumptions \ref{assu:surj} and \ref{assu:reg}}  hold true,  \eqref{eq:vi} is equivalent to  \eqref{eq:adjoint1}--\eqref{eq:grad}. \end{theorem}

\section{{Formulation of viscous history-dependent EVIs as non-smooth ODEs}}\label{sec:evi}
This section focuses on proving that  the following viscous history-dependent evolution
\begin{equation}\label{eq:evi}\tag{EVI}
 R(\HH(y)(t),\eta)-R(\HH(y)(t),\dot y(t))+\dual{\V\dot y(t)}{\eta-\dot y(t)}_Y  \geq \dual{g(y(t),\ell(t))}{\eta-\dot y(t)}_Y \quad \forall\, \eta \in Y, 
\end{equation}\text{\ae}$(0,T)$, is equivalent to a non-smooth  ODE in the Hilbert space $Y$, cf.\,Theorem \ref{thm:ode} below. 

If the dependency of the dissipation $R$ on the history operator $\HH$ is dropped, \eqref{eq:evi} is just the viscous EVI from \cite[Sec.\,3]{st_coup}.
By contrast to \cite{st_coup}, the non-smooth non-linearity $\FF$ (Definition \ref{def:fp} below) appearing in our  ODE \eqref{eq:ode} has two arguments $(\zeta,\omega)$ such that for each $\zeta$, $\FF(\zeta)$  is the solution operator of an elliptic VI of the second kind, cf.\,\eqref{eq:vi_ell}. This can be described by means of an explicit formula featuring the  projection operator \cite[Sec.\,3]{st_coup}. 
Such a formula allows us to state conditions under which the directional differentiability  of the solution map associated to \eqref{eq:evi} is guaranteed (Theorem \ref{thm:ode_dir} below).

In all what follows, $\ell \in \lH$ is fixed. 
Here, $Z$ is a real reflexive Banach space, while $Y$ is a real Hilbert space.

\begin{assumption}\label{assu:st}
For the operators  in \eqref{eq:evi} we require:
 \begin{enumerate}
   \item\label{it:st1} The non-smooth functional $R:X \times Y  \to (-\infty,\infty]$ has the following properties:
 \begin{enumerate}
   \item\label{it:st11} For each $\zeta \in X$, $R(\zeta, \cdot)$ is proper, convex, lower semicontinuous and \bl{positively} homogeneous, i.e., $R(\zeta,\alpha \eta)=\alpha R(\zeta,\eta)$ for all $\alpha >0$ and all $\eta \in Y$.
   \item \label{it:st12} There exists $L_R \geq 0$ such that 
\begin{align*}
R(\zeta_1,\eta_2)-& R(\zeta_1,\eta_1)+R(\zeta_2,\eta_1)-R(\zeta_2,\eta_2)  \leq L_R\,\|\zeta_1-\zeta_2\|_X\|\eta_1-\eta_2\|_Y 
   \\& \quad \qquad  \qquad \forall\,\zeta_1,\zeta_2 \in X,\ \forall\,\eta_1 \in  \dom R(\zeta_1,\cdot),\eta_2 \in \dom R(\zeta_2,\cdot).
   \end{align*}
   \end{enumerate}
   \item\label{it:st_h}  The \textit{history operator} $\HH:\lly \to \llx$ satisfies 
$$\|\HH (\eta_1)(t)-\HH (\eta_2)(t)\|_{X} \leq L_\HH\, \int_0^t \|\eta_1(s)-\eta_2(s)\|_{Y} \,ds \quad  \text{a.e.\ in } (0,T),$$for all  $\eta_1,\eta_2 \in \lly$, where $L_\HH>0$ is a positive constant.   Moreover, $ \HH:  \lly \to \llx$ is supposed to be directionally differentiable.
  \item\label{it:st2} The viscosity operator $\V \in \LL(Y,Y^\ast)$ is coercive, i.e., there exists $\vartheta >0$ so that 
  $\dual{\V\eta}{\eta}_Y \geq \vartheta \|\eta\|_Y^2$ for all $\eta \in Y$. Moreover, $\V$ is \bl{self-adjoint}, i.e., $\dual{\V \eta}{y}_Y=\dual{\V y}{\eta}_Y$ for all $\eta, y \in Y$. 
    \item\label{it:st3}
The  mapping  $g: Y \times Z \to Y^\ast$ is {directionally} differentiable and   Lipschitz continuous with Lipschitz constant $L_g>0$.
\end{enumerate}
\end{assumption}

\begin{remark}\label{rem:volt}
The requirement in Assumption \ref{assu:st}.\ref{it:st12} corresponds to \cite[(3.54)]{sofonea}. Note that this condition is needed in the proof of \cite[Thm.\,4.9, page 72]{sofonea} as well, in order to be able to apply a fix point argument, which in turn leads to the existence of a unique solution for \eqref{eq:evi}.
\\
Assumption \ref{assu:st}.\ref{it:st_h} is satisfied by the Volterra operator $\HH:\lly \to C([0,T]; X)$, defined as 
\begin{equation*}
 [0,T] \ni t \mapsto \HH(y)(t):=\int_0^t A(t-s)y(s) \,ds +y_0 \in X,\end{equation*}
where $A\in C([0,T]; \LL(Y,X))$ and $y_0\in X.$
This type of operator is often employed in the study of history-depedent evolutionary variational inequalities, see e.g.\cite[Ch.\,4.4]{sofonea}.\end{remark}

\subsection{Preliminaries}

 In the sequel, Assumption \ref{assu:st} is tacitly assumed, without mentioning it every time. Note that, in view of Assumption \ref{assu:st}.\ref{it:st2}, the operator $\V$ induces a norm on $Y$, which will be denoted by $\|\cdot\|_{\V}:=\sqrt{\dual{\V\cdot }{\cdot }_Y}$. Similarly,  the operator $\V^{-1}$ induces a norm on $Y^\ast$, which we abbreviate $\|\cdot\|_{\V^{-1}}:=\sqrt{\dual{\V^{-1}\cdot }{\cdot }_{Y^\ast}}$ in the following. We remark that $\|\cdot\|_{\V}$ and $\|\cdot\|_{\V^{-1}}$ are equivalent to $\|\cdot\|_{Y}$ and $\|\cdot\|_{Y^\ast}$, respectively.

 \begin{definition}[The projection operator]\label{proj}
Let us define the function $B:X \times Y^\ast \to Y^\ast$ as  \begin{equation*}\label{eq:b}
B (\zeta,\o):=P_{\partial_2 R(\zeta,0)}\o,
\end{equation*}where, for each $\zeta \in X$, the operator $P_{\partial_2 R(\zeta, 0)}:Y^\ast \to Y^\ast $ is the \bl{(metric)} projection onto the set $$\partial_2 R(\zeta, 0)=\{\varphi \in Y^\ast: \dual{\varphi}{v}_Y \leq R(\zeta,v) \quad \forall\,v \in Y\}$$ w.r.t.\ the \bl{inner product $\dual{\V^{-1}\cdot }{\cdot }_{Y^\ast}$}, i.e.,  $P_{\partial_2 R(\zeta, 0)}\o$ is the unique solution of \begin{equation}\label{eq:min} \min_{\mu \in \partial_2 R(\zeta, 0)} \frac{1}{2}\|\o-\mu\|_{\V^{-1}}^2
\end{equation} 
for any $\o \in Y^\ast$.
\end{definition}

\begin{definition}[The non-smooth non-linearity]\label{def:fp}
Let us define the function $\FF:X \times Y^\ast \to Y$ as  \begin{equation}\label{eq:f}
\mathcal{F} (\zeta,\o):=\V^{-1}(\o-B(\zeta,\o)).
\end{equation}
\end{definition}

\begin{lemma}\label{lem:f}
For each $\zeta \in X$, the  mapping $\FF(\zeta,\cdot):Y^\ast \ni \o \mapsto z \in Y$ is the solution operator of the following elliptic VI of the second kind
  \begin{equation}\label{eq:vi_ell}
R(\zeta,\eta)-R(\zeta,z)+\dual{\V z}{\eta-z}_Y \geq \dual{\o}{\eta-z}_Y  \quad \forall\, \eta \in Y.
\end{equation}Thus, \eqref{eq:vi_ell} is equivalent to $z=\FF(\zeta,\o)=\V^{-1}(\o-B(\zeta,\o))$ for any $(\zeta,\o) \in X \times Y^\ast$.
\end{lemma}
\begin{proof}
The result follows by applying \cite[Lemma 3.3]{st_coup} for $R(\zeta,\cdot)$ for each $\zeta \in X$.
\end{proof}




\begin{lemma}[Lipschitz continuity of $\FF$]\label{lip_F}
The function $\FF:X \times Y^\ast \to Y$ is Lipschitz continuous with Lipschitz constant $L_\FF:=\frac{\max\{1,L_R\}}{\vartheta}$.
\end{lemma} 
\begin{proof}Let $(\zeta_1,\o_1),(\zeta_2,\o_2) \in X \times Y^\ast$ be arbitrary but fixed and let us abbreviate $z_i:=\FF(\zeta_i,\o_i),\ i=1,2.$ According to Lemma \ref{lem:f}, $z_i$, $i=1,2$, solves the VI
\[  R(\zeta_i,\eta)-R(\zeta_i,z_i)+\dual{\V z_i}{\eta-z_i}_Y \geq \dual{\o_i}{\eta-z_i}_Y  \quad \forall\, \eta \in Y.\]
Testing with $z_j$, $j \neq i$ and adding the resulting inequalities leads to 
\begin{align*} 
 \dual{\V (z_2-z_1)}{z_2-z_1}_Y &\leq \dual{\o_2-\o_1}{z_2-z_1}_Y
\\&\quad + R(\zeta_1,z_2)-R(\zeta_1,z_1)+R(\zeta_2,z_1)-R(\zeta_2,z_2)
\\&\leq \|\o_2-\o_1\|_{Y^\ast}\|z_2-z_1\|_Y
+L_R \|\zeta_2-\zeta_1\|_{X}\|z_2-z_1\|_Y,\end{align*}where the last inequality is due to  Assumption \ref{assu:st}.\ref{it:st12}; note that $z_i \in \dom R(\zeta_i,\cdot),\ i=1,2,$ as a result of Lemma \ref{lem:f}.
Now, the coercivity of $\V$, see Assumption \ref{assu:st}.\ref{it:st2}, yields the desired assertion.
\end{proof}

As an immediate consequence of Lemma \ref{lem:f} and Assumption \ref{assu:st}.\ref{it:st2}, we have the following 

\begin{corollary}
\label{cor:fp}
The solution operator $\FF:X \times Y^\ast \ni (\zeta, \o) \mapsto z \in Y$ of \eqref{eq:vi_ell} is directionally differentiable at $(\bar \zeta,\bo) \in X \times Y^\ast$ if and only if the mapping $B:( \zeta,\o)\mapsto P_{\partial_2 R(\zeta, 0)}\o$ is directionally differentiable at $(\bar \zeta,\bo) \in X \times Y^\ast$. If this is the case, then
 \begin{equation*}\label{eq:f_deriv}
\FF'((\bar \zeta,\bo);(\delta \zeta,\delta \omega))=\V^{-1}\Big(\doo-B'((\bar \zeta,\bo);(\delta \zeta,\delta \omega))\Big) \quad \forall\, (\delta \zeta,\doo) \in X \times Y^\ast.\end{equation*}
\end{corollary}

\subsection{Main results}
\begin{theorem}[Viscous history-dependent EVIs are non-smooth ODEs]\label{thm:ode}
The viscous history-dependent problem \eqref{eq:evi} is equivalent to the following ODE in Hilbert space
 \begin{equation}\label{eq:ode}
\dot y= \FF(\HH( y),g(y,\ell)) \ \text{ \ae }(0,T),
 \end{equation}where $\FF$ is given by \eqref{eq:f} and  $\ell :[0,T] \to Z$. If $y(0)=0$, then \eqref{eq:evi} admits a unique solution $y \in \hy$ for every right-hand side $\ell \in \lH$. \end{theorem}
 \begin{proof}
 The first assertion is due to Lemma \ref{lem:f}.  To solve \eqref{eq:ode}, we apply a fixed-point argument. For this, we take a look at the 
 mapping $L^2(0,t^\ast;Y) \ni \eta \mapsto \GG(\eta) \in H^1(0,t^\ast;Y)$, given by 
\begin{equation*}\label{eq:FF}
\GG(\eta)(\tau):=\int_0^\tau  \FF(\HH( \eta),g(\eta,\ell))(s) \;ds \quad \forall\, \tau \in [0,t^\ast],
\end{equation*}
where $t^\ast \in (0,T]$ is to be determined so that $\GG:L^2(0,t^\ast;Y) \to L^2(0,t^\ast;Y)$ is a contraction. 
For all $\eta_1,\eta_2 \in L^2(0,t^\ast; Y)$ the following estimate is true
\begin{equation}\label{eqq}
\begin{aligned}
 \|\GG(\eta_1)(\tau)&-\GG(\eta_2)(\tau)\|_{Y} 
 \\&\quad \leq L_{\FF} \int_0^\tau   \|\HH(\eta_1)(s)-\HH(\eta_2)(s)\|_{X} +\|g(\eta_1(s),\ell(s))-g(\eta_2(s),\ell(s))\|_{Y^\ast} \,ds
 \\& \quad    \leq L_{\FF}\int_0^\tau L_{\HH} \int_0^s \|\eta_1(\zeta)-\eta_2(\zeta) \|_{Y} \,d\zeta  +L_g\,\|\eta_1(s)-\eta_2(s) \|_{Y} \,ds
\\ &\quad \leq t^\ast L_{\FF}  \, L_{\HH}\|\eta_1-\eta_2\|_{L^1(0,t^\ast;Y)}
+L_{\FF}   \, L_{g}\|\eta_1-\eta_2\|_{L^1(0,t^\ast;Y)}
\\ &\quad \leq L_{\FF}  \,(t^\ast)^{1/2} (t^\ast L_{\HH} +L_g) \|\eta_1-\eta_2\|_{L^2(0,t^\ast;Y)}
\quad  \text{ for all }\tau \in [0,t^\ast].
\end{aligned}\end{equation}
Here we used the fact that $\FF:X \times Y^\ast \to Y$ is Lipschitzian according to Lemma \ref{lip_F}, as well as Assumptions \ref{assu:st}.\ref{it:st_h} and \ref{assu:st}.\ref{it:st3}. From \eqref{eqq} we deduce
 \begin{equation*}\label{eq:kontrak}
 \|\GG(\eta_1)-\GG(\eta_2)\|_{L^2(0,t^\ast; Y)}
\leq L_{\FF}  \,t^\ast (t^\ast L_{\HH} +L_g) \|\eta_1-\eta_2\|_{L^2(0,t^\ast;Y)},\end{equation*}which allows us to conclude that $\GG$ is a contraction for a small enough $t^\ast$. Thus, the ODE \eqref{eq:ode} restricted on $(0,t^\ast)$ admits a unique solution in $H^1_0(0,t^\ast;Y)$. Now, the unique solvability of \eqref{eq:ode} on the whole interval $(0,T)$ and the desired regularity of $y$ follows by a concatenation argument.
\end{proof}
 \begin{proposition}[Lipschitz continuity]\label{prop:lip}
The  solution map $\SS: \lH \ni \ell \mapsto y \in H^1_0(0,T;Y)$ associated to \eqref{eq:evi} is Lipschitz continuous.     \end{proposition}
 \begin{proof}
Let $\ell_1,\ell_2 \in \lH$ be arbitrary but fixed and let us abbreviate $y_i:=\SS(\ell_i), i=1,2$. Estimating as in \eqref{eqq} leads to 
\begin{equation*}\label{eqqq}
\begin{aligned}
\| &y_1(t)- y_2(t)\|_Y \leq  \int_0^t \| \FF(\HH( y_1),g(y_1,\ell_1))(s)-\FF(\HH( y_2),g(y_2,\ell_2))(s)\|_Y \,ds
\\&\leq L_{\FF}  \int_0^t L_{\HH} \int_0^s \|y_1(\zeta)-y_2(\zeta) \|_{Y} \,d\zeta  +L_g\,(\|y_1(s)-y_2(s) \|_{Y}+\|\ell_1(s)-\ell_2(s) \|_{Z}) \,ds 
\end{aligned}\end{equation*}\text{\ae }$(0,T)$.
Gronwall's lemma and arguing as at the end of the proof of Theorem \ref{thm:ode} then yields the desired assertion.
\end{proof}

\begin{lemma}
For all $\eta, \delta \eta_1, \delta \eta_2 \in \lly$, it holds
\begin{equation}\label{eq:h'}
\|\HH'(\eta;\delta \eta_1)(t)-\HH'(\eta;\delta \eta_2)(t)\|_{X} \leq L_{\HH}  \int_0^t \|\delta \eta_1(s)-\delta \eta_2(s)\|_{Y} \,ds
\end{equation} \ae $(0,T)$.\end{lemma}
\begin{proof}
We observe that, in view of Assumption \ref{assu:st}.\ref{it:st_h}, it holds
$$\frac{1}{\tau}\|\HH (\eta+ \tau \delta \eta_1)(t)-\HH (\eta+\tau \delta \eta_2)(t)\|_{X} \leq   L_\HH\, \int_0^t \|\delta \eta_1(s)-\delta \eta_2(s)\|_{Y} \,ds $$$\ae $  $(0,T)$, for all  $\eta, \delta \eta_1,\delta \eta_2 \in \lly$ and all $\tau>0$. Passing to the limit $\tau \searrow 0$, where one uses the  directional differentiability of $ \HH$ and  the fact that convergence in $\llx$ implies $\ae$ convergence in $X$ for a subsequence,
then  yields  
the desired estimate.\end{proof}

\begin{theorem}[Hadamard directional differentiability]\label{thm:ode_dir}
The  solution map $\SS: \lH \ni \ell \mapsto y \in H^1_0(0,T;Y)$ associated to \eqref{eq:evi} is Hadamard directionally differentiable \cite[Def.\ 3.1.1]{schirotzek} at $\ll \in \lH$, if $\FF:X \times Y^\ast \to Y$  is directionally differentiable at $(\HH( \bar y)(t),g(\bar y(t),\ll(t)))$ f.a.a.\ $t \in (0,T)$ or, equivalently, if $B:X \times Y^\ast \to Y^\ast$ does so, where we abbreviate $\bar y:=\SS(\ll)$.
Its directional derivative $\dy:=\SS'(\ll;\dl)$ at $\ll$ in direction $\dl \in \lH$ is the unique solution of
 \begin{equation} \label{eq:ode_lin}
   \dot \dy= \FF'\Big((\HH(\bar y),g(\bar y,\ll));\big(\HH'(\bar y;\dy),g'(\yy,\ll){;}(\dy,\dl)\big)\Big) \  \text{ \ae }(0,T),\ \delta y(0) = 0.\end{equation}
      \end{theorem}
\vspace{-0.4cm}\begin{proof}
By arguing as in the  proof of Theorem \ref{thm:ode}, we get that, for any $\dl \in \lH$, \eqref{eq:ode_lin} admits a unique solution $\dy \in \hy$. Note that in this case we rely on the Lipschitz continuity of the directional derivatives of $\FF$ and $g$ w.r.t.\ direction and on the estimate \eqref{eq:h'}. 
Further, since $\FF$ is Lipschitzian according to Lemma \ref{lip_F}, we can apply Lebesgue's dominated convergence theorem to obtain that  \[\FF:\llx \times L^2(0,T;Y^\ast) \to \lly\] is directionally differentiable at $(\HH(\bar y),g(\bar y,\ll))$. Moreover, by relying again on Lemma \ref{lip_F}, 
we obtain that $\FF$ is even Hadamard directionally differentiable at $(\HH(\bar y),g(\bar y,\ll))$ \cite[Def.\ 3.1.1]{schirotzek}, as a result of \cite[Lem.\ 3.1.2(b)]{schirotzek}.  Since 
$$f:\lly \times L^2(0,T;Z) \ni  (y,\ell) \mapsto (\HH(y),g(y,\ell)) \in \llx \times L^2(0,T;Y^\ast)$$ is directionally differentiable, by Assumptions \ref{assu:st}.\ref{it:st_h} and  \ref{assu:st}.\ref{it:st3}, chain rule \cite[Prop.\ 3.6(i)]{shapiro} implies that \vspace{-0.4cm} \[\widehat G:=\FF \circ f\]
 is (Hadamard) directionally differentiable at $(\yy,\ll)$ with 
  \begin{equation}\label{eq:g_d}
  \widehat G'((\yy,\ll);h)=\FF'\Big((\HH(\bar y),g(\bar y,\ll));\big(\HH'(\bar y;h_1),g'((\yy,\ll){;}h)\big)\Big) \quad \forall\, h=(h_1,h_2) \in  L^2(0,T;Y \times Z). \end{equation}
For simplicity, in the following we abbreviate $\bar y^{\tau}:=\SS(\ll + \tau \,\dl)$, where $\tau>0$ and $\dl \in \lH$ are arbitrary, but fixed. Due to \eqref{eq:g_d} and by combining the equations for $\bar y^{\tau}$, $\bar y$ and \eqref{eq:ode_lin}, we obtain
  \begin{equation}\label{eq:www} \begin{aligned}
 \frac{d}{dt} \Big(\frac{\bar y^\tau-\bar y}{\tau}-\dy \Big)&=\frac{\widehat G(\bar y^\tau,\ll+ \tau \,\dl) -\widehat G(\bar y,\ll ) }{\tau}
\\&\qquad \qquad -\widehat G'\big((\yy,\ll);(\dy,\dl)\big) \quad \text{\ae} (0,T),  
 \\  \Big(\frac{\bar y^\tau-\bar y}{\tau}-\dy\Big)(0) &= 0.
 \end{aligned}
 \end{equation}This implies  
   \begin{equation}\label{eq:awalin-}
 \begin{aligned}
&\Big\|\Big(\frac{\bar y^\tau-\bar y}{\tau}-\dy \Big)(t)\Big\|_Y
\\  &\  \leq \int_0^t \Big\|\frac{\widehat G (\bar y^\tau,\ll + \tau \,\dl)(s) -\widehat G\big((\bar y,\ll )+\tau (\dy,\dl) \big)(s) }{\tau}\Big\|_Y \,ds
\\ +& \int_0^t \underbrace{\Big\|\frac{\widehat G\big((\bar y,\ll )+\tau (\dy,\dl) \big)(s)-\widehat G(\bar y,\ll )(s)}{\tau}
 -\widehat G'\big((\yy,\ll);(\dy,\dl)\big) (s)\Big\|_Y }_{=:A_\tau(s)} \,ds
 \\  &\  \leq \frac{ L_\FF}{\tau} \int_0^t \, \Big\|{f (\bar y^\tau,\ll + \tau \,\dl)(s) -f \big((\bar y,\ll )+\tau (\dy,\dl) \big)(s) }\Big\|_{X \times Y^\ast}
  \\&\qquad + \int_0^t A_\tau(s) \,ds 
   \\  &\  \leq \frac{ L_\FF}{\tau} \int_0^t \, \|\HH(\bar y^\tau)(s)-\HH(\bar y +\tau \dy)(s)\|_{X} \,ds
   \\&\qquad + \frac{ L_\FF}{\tau} \int_0^t\|g(\bar y^\tau(s),(\ll + \tau \,\dl)(s))-g((\bar y +\tau \dy)(s),(\ll + \tau \,\dl)(s))\|_{Y^\ast} \,ds
  \\&\qquad + \int_0^t A_\tau(s) \,ds 
     \\  &\  \leq \frac{ L_\FF}{\tau} \int_0^t \,  L_{\HH} \int_0^s \|\bar y^\tau(\zeta)-\bar y(\zeta)-\tau \dy (\zeta) \|_{Y}\,d\zeta   +  L_g\,\|\bar y^\tau(s)-\bar y(s)-\tau \dy  (s) \|_{Y}\,ds
  \\&\qquad + \int_0^t A_\tau(s) \,ds
       \\  &\  \leq (T\,{ L_\FF}  L_{\HH} +L_g) \int_0^t \Big\|\frac{\bar y^\tau(s)-\bar y(s)-\tau \dy  (s) }{\tau}  \Big\|_{Y}\,ds + \int_0^t A_\tau(s) \,ds \qquad \forall\,t \in [0,T],
 \end{aligned}
 \end{equation}
where we employed the definitions of $\widehat G$, of $f$, and  the fact that $\FF:X \times Y^\ast \to Y$ is Lipschitzian according to Lemma \ref{lip_F}, as well as Assumptions \ref{assu:st}.\ref{it:st_h} and \ref{assu:st}.\ref{it:st3}.
Applying Gronwall's inequality in \eqref{eq:awalin-} then yields 
   \begin{equation}\label{eq:wwww}
\Big\|\Big(\frac{\bar y^\tau-\bar y}{\tau}-\dy \Big)(t)\Big\|_Y \leq c\,\int_0^t A_\tau(s) \,ds \qquad \forall\,t \in [0,T],
 \end{equation}where $c>0$ is a constant dependent only on the given data. Now, \eqref{eq:www}  and estimating  as in \eqref{eq:awalin-}, in combination with \eqref{eq:wwww}, leads to    \begin{equation}\label{eq:ww}
\Big\|\frac{\bar y^\tau-\bar y}{\tau}-\dy \Big\|_{H^1(0,T;Y)} \leq c\,\|A_\tau\|_{L^2(0,T)} \quad \forall\,\tau >0.
 \end{equation}
 On the other hand, we recall the definition of $A_\tau$ in  \eqref{eq:awalin-} and the fact that $\widehat G:\lly \times L^2(0,T;Z) \to \lly$ is directionally differentiable at $(\yy,\ll)$, from which we deduce $$\|A_\tau\|_{L^2(0,T)} \to 0 \quad \text{as } \tau \searrow 0.$$
Finally, the desired assertion follows from \eqref{eq:ww}. The Hadamard directional differentiability \cite[Def.\ 3.1.1]{schirotzek} is due to Proposition \ref{prop:lip} and  \cite[Lem.\ 3.1.2(b)]{schirotzek}. This completes the proof.  \end{proof}

\begin{remark}\label{p}
All the results established in this subsection are valid for right-hand sides $\ell \in L^p(0,T;Z)$, where $1 \leq p \leq \infty$, in which case the unique solution to \eqref{eq:ode} belongs to $W^{1,p}_0(0,T;Y)$ (provided that the history operator $\HH$ maps between $L^p(0,T)$ spaces).
\end{remark}

\section{Strong stationarity for the control of viscous damage models with fatigue}\label{sec:dm}
Based on the results from the previous sections, we next derive strong stationary optimality conditions for the  control of two viscous damage models with fatigue. 

The underlying \textit{non-viscous} damage problem with fatigue reads as follows:
\begin{equation}\label{eq:n}
\left.
  \begin{gathered}
-\partial_q \EE(t,q(t)) \in \partial_2 \RR (H(q)(t),\dot{q}(t)) \quad \text{in }\hoo , \quad q(0) = 0
  \end{gathered}\ \right\}
  \end{equation}a.e.\ in $(0,T)$, cf\, \cite{alessi}.    In \eqref{eq:n}, $\EE:[0,T] \times H^{1}(\Omega) \rightarrow \mathbb{R}$ is the stored energy; this will be specified in the upcoming subsections, depending on the setting.
   The non-viscous dissipation   $\RR:L^2(\O) \times  H^1(\O) \rightarrow (-\infty,\infty]$ is defined as \begin{equation}\label{def:rrr}
\RR(\zeta,\eta):=\left\{ \begin{aligned}\int_\Omega \kappa(\zeta) \,  \eta \;dx , &\quad \text{if }\eta \geq 0 \text{ a.e. in }\Omega,
\\\infty  &\quad \text{otherwise.}\end{aligned} \right.
\end{equation}

The differential inclusion appearing in \eqref{eq:n} describes the evolution of the damage variable $q$
under \textit{fatigue} effects. Therein, $H$ is a so-called \textit{history operator}  that models how the  damage experienced by the material affects its fatigue level. 
The fatigue degradation mapping $\kappa:\R \to \R$ appearing in \eqref{def:rrr} indicates in which measure the fatigue affects the fracture toughness  of the material.  Whereas usually the toughness of the material  is described by a fixed  (nonnegative) constant \cite{fn96, FK06}, in the present model it changes at each point in time and space, depending on $H(q)$. To be more precise, the value of the fracture toughness of the body at $(t,x)$ is given by  $\kappa(H(q))(t,x)$, cf.\ \eqref{def:rrr}. Hence, the model \eqref{eq:n} takes into account the following crucial aspect: the occurrence of damage is favoured in regions where fatigue accumulates.

\begin{assumption}\label{assu:standd}For the mappings associated with fatigue  in \eqref{eq:n} we require the following:
 \begin{enumerate}
  \item\label{it:stand1d}  The \textit{history operator} $H:\llz \to \llz$ satisfies 
$$\|H(y_1)(t)-H (y_2)(t)\|_{\lo} \leq L_H\, \int_0^t \|y_1(s)-y_2(s)\|_{\lo} \,ds \quad \ae  (0,T),$$for all  $y_1,y_2 \in \llz$, where $L_H>0$ is a positive constant.   {Moreover, $ H:  \llz \to \llz$ is supposed to be G\^ateaux-differentiable.}
\item\label{it:stand2d}
The non-linear function  $\kappa: \R \to \R$ is assumed to be Lipschitz continuous with Lipschitz constant $L_\kappa>0$ and  differentiable.  
 \end{enumerate}
\end{assumption} 

\subsection{$H^1_0-$viscosity}\label{sec:0}
The model we intend to examine describes the evolution of damage under the influence of a time-dependent load $\ell:[0,T] \to H^{-1}(\Omega)$ (control) acting on a body occupying the bounded Lipschitz domain $\O \subset \R^N$, $N \in \{2,3\}$. The induced  damage is expressed in terms of $q : [0,T] \to H^{1}_0(\Omega)$. 
The problem we consider is a viscous version of the fatigue damage model addressed in \cite{alessi}; for simplicity reasons, we do not take a displacement variable into account. Let us mention that $H_0^1-$ viscosity has been used in the context of optimal control  in \cite{sww} as well, see \cite[Eq.\,(5)]{sww}, where a sweeping process is considered. For more details on the viscous approximation of damage models we refer the reader to \cite[Sec.\,4.4]{mielke}.

The viscous (single-field) damage problem with fatigue reads as follows:
\begin{equation}\label{eq:q1}
-\partial_q \EE(t,q(t)) \in \partial_2 \RR_\epsilon (H(q)(t),\dot{q}(t))\quad \text{in }\hoon  , \quad q(0) = 0
  \end{equation}a.e.\ in $(0,T)$.   In \eqref{eq:q1}, the stored energy $\EE:[0,T] \times H_0^{1}(\Omega) \rightarrow \mathbb{R}$ is given by
\begin{equation}\label{eq:e1}\begin{aligned}
\EE(t, q)&:=\frac{\alpha}{2}\|\nabla q\|_{L^2(\O)}^2-\dual{\ell(t)}{q}_{H^{1}(\Omega)},
\end{aligned}\end{equation}where $\alpha>0$ is a fixed parameter.
  The viscous dissipation  $\RR_\epsilon:L^2(\O) \times  H^1_0(\O) \rightarrow (-\infty,\infty]$ is defined as \begin{equation}\label{def:r1}
\RR_\epsilon(\zeta,\eta):=\left\{ \begin{aligned}\int_\Omega \kappa(\zeta) \,  \eta \;dx +\frac{\epsilon}{2}\|\nabla \eta\|^2_{\lo}, &\quad \text{if }\eta \geq 0 \text{ a.e. in }\Omega,
\\\infty  &\quad \text{otherwise,}\end{aligned} \right.
\end{equation}
where  $\epsilon>0$ is the viscosity parameter.
\begin{definition}\label{def:c}
The polyhedric set $\CC \subset \hon$ is defined as 
 $$\CC:=\{v \in \hon: v\geq 0 \text{ \ae} \O\}.$$
\end{definition}

In all what follows, the \bl{(metric)} projections onto a subset of $\hoon$ are considered w.r.t.\ the \bl{inner product $\dual{\frac{1}{\epsilon}(-\Delta)^{-1}\cdot }{\cdot }_{\hoon}$}, unless otherwise specified; cf.\,also \eqref{eq:min}.

\begin{lemma}\label{lem:proj}
For each $(\zeta, \omega) \in \lo \times \hoon$ it holds
$$P_{\partial_2 \RR(\zeta,0)}\omega=P_{\partial I_\CC (0)}(\omega-\kappa(\zeta))+\kappa(\zeta) \quad \text{in }\hoon.$$
\end{lemma}
\begin{proof}
We observe that
\begin{equation}\label{subdifff}
\begin{aligned}
\partial_2 \RR(\zeta,0)&=\{\mu \in \hoon | \,  \dual{\mu}{v}_{\hon} \leq \RR(\zeta,v) -\RR(\zeta,0) \quad \forall\,v\in \hon\}
\\&=\{\mu \in \hoon | \,  \dual{\mu-\kappa(\zeta)}{v}_{\hon} \leq 0 \quad \forall\,v\in \hon \text{ with } v\geq 0 \text{ a.e.\ in }\O\}
\\&=\partial I_\CC(0)+\kappa(\zeta) \quad \forall\,\zeta \in \lo,\end{aligned}\end{equation}where the second identity in \eqref{subdifff} is due to \eqref{def:rrr}. Moreover, given a closed convex set $M \subset \hoon$, it holds $$P_M \omega=P_{M-\psi}(\omega -\psi)+\psi \quad \forall\,\o,\psi \in \hoon, $$which can be easily checked by employing the definition of the projection operator. The desired assertion now follows from \eqref{subdifff}.
\end{proof}
\begin{proposition}[Control-to-state map]\label{lem:ode1}
For every $\ell \in \llun$, the viscous damage problem with fatigue \eqref{eq:q1} admits a unique solution $q \in   \hhyy $, which is characterized by  \begin{equation}\label{eq:syst_diff1}
 \dot q(t)   = \frac{1}{\epsilon} (-\Delta)^{-1}(\mathbb{I}-P_{\partial I_\CC (0)})(\alpha \Delta q(t) +\ell(t)-\F(q)(t)) \quad \text{in }\hon,  \quad q(0) = 0
 \end{equation} a.e.\ in $(0,T)$.
   \end{proposition} 
 \begin{proof}
 In view of \eqref{def:r1}, \eqref{def:rrr} and the sum rule for convex subdifferentials, it holds
\[\partial_2 \RR_\epsilon (H(q)(t),\dot{q}(t))=\partial_2 \RR (H(q)(t),\dot{q}(t))-\epsilon\, \Delta \dot{q}(t)\]
a.e.\ in $(0,T)$. Thus, on account of \eqref{eq:e1}, the evolution \eqref{eq:q1} is  equivalent to 
    \begin{equation}\label{eq:q_red}  
    \RR(H(q)(t),v)-\RR(H(q)(t),\dot{q}(t))+\epsilon\,(\nabla \dot{q}(t),\nabla (v-\dot{q}(t)))_{\lo}  \geq \dual{\alpha \Delta q(t) +\ell(t)}{v-\dot{q}(t)}_{H_0^1(\O)}          \end{equation}for all $v \in \hon$.
Now, with the notations from section \ref{sec:evi}, we see that if we set
       \begin{subequations}\label{eq:assu}\begin{gather}
        X:=\lo,\quad Y:=\hon,\quad Z:=\hoon, \\
       R:=\RR, \quad \HH:=H,\quad \V:=-\epsilon\, \Delta, \quad g( \widetilde q, \widetilde \ell):=\alpha \Delta \widetilde q+\widetilde \ell, \end{gather}
 \end{subequations}then \eqref{eq:evi} coincides with \eqref{eq:q_red}. Indeed, the quantities in \eqref{eq:assu} satisfy Assumption \ref{assu:st}, as we will next see. While Assumption \ref{assu:st}.\ref{it:st11} can be easily checked, the condition in Assumption \ref{assu:st}.\ref{it:st12} is due to
 \begin{align*}
R(\zeta_1,\eta_2)-& R(\zeta_1,\eta_1)+R(\zeta_2,\eta_1)-R(\zeta_2,\eta_2)
\\&= \int_\Omega \kappa(\zeta_1) \,(\eta_2-\eta_1) \;dx-\int_\Omega \kappa(\zeta_2) \,(\eta_2-\eta_1) \;dx
\\&\quad \leq L_{\kappa} \,\|\zeta_1-\zeta_2\|_{\lo}\|\eta_1-\eta_2\|_{\hon} 
   \\& \quad \quad \qquad  \qquad \forall\,\zeta_1,\zeta_2 \in \lo,\ \forall\,\eta_1, \eta_2 \in  \CC.
   \end{align*}Note that $\dom R(\zeta,\cdot)=\CC$ for all $\zeta \in \lo$ and that the above inequality follows from Assumption \ref{assu:stand}.\ref{it:stand2d}. Further, in view of $\hon \embed \lo$ and Assumption \ref{assu:stand}.\ref{it:stand1d}, we see that Assumption \ref{assu:st}.\ref{it:st_h} is satisfied as well. Moreover, we observe that Assumption \ref{assu:st}.\ref{it:st2} (with $\vartheta=\epsilon$) and Assumption \ref{assu:st}.\ref{it:st3} also hold true in the setting \eqref{eq:assu}.

 Since the quantities in \eqref{eq:assu} satisfy the entire Assumption \ref{assu:st} and since the evolution in \eqref{eq:q1} is equivalent to \eqref{eq:q_red} we immediately obtain from Theorem \ref{thm:ode} the unique solvability of \eqref{eq:q1} and the desired regularity of the solution. Moreover, according to Theorem \ref{thm:ode}, \eqref{eq:q_red} is equivalent to
   \begin{equation*}
   \dot q(t)=\FF \big(H(q)(t),g(q,\ell)(t)\big) \ \text{ a.e.\ in }(0,T),\end{equation*}
   where 
      \begin{equation}\label{FF}
   \FF(\zeta, \omega)=\frac{1}{\epsilon} (-\Delta)^{-1}(\mathbb{I}-P_{\partial_2 \RR(\zeta,0)})\omega \quad \forall\,  (\zeta,\omega) \in \lo \times \hoon.
   \end{equation}
    That is, \eqref{eq:q_red}, and thus, the evolution in \eqref{eq:q1}, can be rewritten as
    \begin{equation}\label{ode}
   \dot q(t)=\frac{1}{\epsilon} (-\Delta)^{-1}(\mathbb{I}-P_{\partial_2 \RR(H(q)(t),0)})\big(g(q,\ell)(t)\big)\ \text{ a.e.\ in }(0,T).\end{equation}
 
 Applying Lemma \ref{lem:proj} now yields that 
 $$P_{\partial_2 \RR(H(q)(t),0)}\big(g(q,\ell)(t)\big)=P_{\partial I_\CC (0)}\big[g(q,\ell)(t)-\F(q)(t)\big]+\F(q)(t),$$
which inserted in \eqref{ode} gives in turn
  \begin{equation*}
   \dot q(t)=\frac{1}{\epsilon} (-\Delta)^{-1}\big(g(q,\ell)(t)-P_{\partial I_\CC (0)}\big[g(q,\ell)(t)-\F(q)(t)\big]-\F(q)(t)\big)\ \text{ a.e.\ in }(0,T).\end{equation*}
In view of the definition of the mapping $g$ in \eqref{eq:assu} we can finally deduce that \eqref{eq:q1} is equivalent to \eqref{eq:syst_diff1}.
 \end{proof}
 \begin{lemma}\label{f_dir}
The mapping $ f: \hoon \to \hon$ defined as $$ f( \omega)= \frac{1}{\epsilon} (-\Delta)^{-1} (\mathbb{I}-P_{\partial I_\CC (0)})(\omega)$$ is directionally differentiable with
$$ f'( \omega;\delta \omega)= \frac{1}{\epsilon} (-\Delta)^{-1} (\mathbb{I}-P_{T(\omega)})(\delta \omega) \quad \forall\,\omega,\delta \omega \in \hoon,$$
where $T(\o):=\overline{\R^+( \partial I_\CC(0 )-P_{\partial I_\CC(0)}\o)}^{\hoon} \cap [\psi(\omega)]^{\perp}$ denotes {the critical cone of $\partial I_\CC(0 ) \subset \hoon$ at $(\omega,\psi(\omega))$} and $\psi(\omega):=\frac{1}{\epsilon} (-\Delta)^{-1}(\omega-P_{\partial I_\CC(0)}\omega).$
 \end{lemma}\begin{proof}
 The assertion is due to the polyhedricity of ${\partial I_\CC (0)}=\CC^\circ \subset \hoon$  (as a result of \cite[Lem.\ 3.2]{wachs} and e.g.\,\cite[Cor.\ 6.46]{bs}) in combination with \cite[Lemma A.1]{st_coup}, which tells us that $P_{\partial I_\CC(0)}:\hoon \to \hoon$ is directionally differentiable with 
$$P_{\partial I_\CC(0)}'(\omega;\delta \omega)=P_{T(\omega)}\delta \omega \quad \forall\,\omega,\delta \omega \in \hoon.$$
 \end{proof}
\begin{proposition}[Directional differentiability]\label{lem:ode1_diff}
The solution map associated to \eqref{eq:q1} $$S: \llun \ni \ell \mapsto q \in   \hhyy$$ is directionally differentiable. 
Its directional derivative $\delta q:=S'(\ell;\dl)$ at the point $\ell  \in \llun$ in direction $\dl \in \llun$ is the unique solution of
\begin{equation}\label{eq:syst_dif}
 \dot {\delta q}(t)   = \frac{1}{\epsilon} (-\Delta)^{-1}(\mathbb{I}-P_{T(z(t))})\Big(\alpha \Delta \delta q(t) +\dl(t)-\F'(q)(\delta q)(t)\Big) \quad \text{in }\hon,  \quad \delta q(0) = 0
 \end{equation} a.e.\ in $(0,T)$. Here, $$T(z(t))=\overline{\R^+( \partial I_\CC(0 )-P_{\partial I_\CC(0 )}z(t))}^{\hoon} \cap [ \dot { q}(t)]^{\perp}$$ denotes {the critical cone of $\partial I_\CC(0 ) \subset \hoon$ at $(z(t),\dot q(t))$}, where
we abbreviate $ z(t):=\alpha \Delta q(t) +\ell(t)-\F(q)(t)$.      \end{proposition}
\begin{proof}According to Theorem \ref{thm:ode_dir}, $S: \llun \ni \ell \mapsto q \in   \hhyy$ is directionally differentiable, if $B: \lo \times \hoon \ni (\zeta, \omega) \mapsto P_{\partial_2 \RR(\zeta,0)}\omega \in \hoon$ does so. Thus, in the light of Lemma \ref{lem:proj} we  need to check that 
\begin{equation}\label{eq:bb}
B: \lo \times \hoon \ni (\zeta, \omega) \mapsto P_{\partial I_\CC (0)}(\omega-\kappa(\zeta))+\kappa(\zeta)  \in \hoon
\end{equation}
 is directionally differentiable. 
 In view of Assumption \ref{assu:standd}.\ref{it:stand2d} and by employing Lebesgue's dominated convergence theorem we obtain that $\kappa:\lo \to \lo$ is G\^ateaux-differentiable. As $P_{\partial I_\CC(0)}:\hoon \to \hoon$ is Lipschitz continuous and directionally differentiable, see  Lemma \ref{f_dir}, we can make use of the chain rule for directionally differentiable functions  \cite[Prop.\ 3.6(i)]{shapiro}. This implies that the mapping $B$ from \eqref{eq:bb} is indeed directionally differentiable with 
\begin{equation}\label{eq:bb_d}
B'((\zeta, \omega);(\delta \zeta, \delta \omega))= P_{T(\omega-\kappa(\zeta))}(\delta \omega-\kappa'(\zeta)(\delta \zeta))+\kappa'(\zeta)(\delta \zeta)\end{equation}
for all $(\zeta,\omega), (\delta \zeta,\delta \omega) \in \lo \times \hoon$.
By Theorem \ref{thm:ode_dir} we then obtain that $S: \llun \ni \ell \mapsto q \in   \hhyy$ is directionally differentiable with directional derivative $\delta q=S'(\ell;\dl)$ satisfying 
 \begin{equation} \label{eq:ode_linn}\begin{aligned}
 &  \dot {\delta q}(t)= \FF'\Big((H(q),g(q,\ell));\big(H'(q)(\delta q),g'((q,\ell){;}(\delta q,\dl))\big)\Big) (t)
   \\&=\frac{1}{\epsilon} (-\Delta)^{-1}[g'((q,\ell){;}(\delta q,\dl))(t)-B'((H(q), g(q,\ell));(H'(q)(\delta q),g'((q,\ell){;}(\delta q,\dl)))(t)] 
   \\&=\frac{1}{\epsilon} (-\Delta)^{-1}[\mathbb{I}   -P_{T(z(t))}][g'((q,\ell){;}(\delta q,\dl))(t)-\kappa'(H(q))(H'(q)(\delta q)) (t)] \quad 
\text{ \ae }(0,T),\end{aligned}\end{equation}
where $\FF$ is given by \eqref{FF} and $g$ is the mapping from \eqref{eq:assu}; recall that $z=\alpha \Delta  q+\ell-\F(q)$ and that $H$ is G\^ateaux-differentiable, by Assumption \ref{assu:standd}.\ref{it:stand1d}. Note that the second identity in \eqref{eq:ode_linn} is due to Corollary \ref{cor:fp}, while the last identity follows from \eqref{eq:bb_d}. Since $g'((q,\ell){;}(\delta q,\dl))=\alpha \Delta \delta q+\delta \ell$ and 
$$\psi(z(t))=\frac{1}{\epsilon} (-\Delta)^{-1}(z(t)-P_{\partial I_\CC(0)}z(t))=\dot q(t) \quad \text{\ae}(0,T),$$
cf.\,\eqref{eq:syst_diff1}, the proof is now complete.
\end{proof}

Next, we want to apply the strong stationarity result from section \ref{sec:pres} to the following
 optimal control problem:

\begin{equation}\tag{P}\label{eq:optprob_q1}
 \left.
 \begin{aligned}
  \min_{\ell \in \llz} \quad & \JJ(q,\ell)
  \\
 \text{s.t.} \quad & q \text{ solves }\eqref{eq:q1} \text{ with r.h.s.\ }\ell.
 \end{aligned}
 \quad \right\}
\end{equation}
In the sequel, the objective $\JJ$ is supposed to fulfill 
\begin{assumption}\label{assu:j_ex1}
  The functional $\JJ: \llu \times \llz \to \R$ 
  is  Fr\'echet-differentiable.\end{assumption}
  
Before stating the strong stationary optimality conditions we establish an estimate which will be useful in the proof of Theorem \ref{thm:ss_qsep1} below.
\begin{lemma}
 There exists a constant $K>0$ dependent only on the given data such that for all $\ell,\dl_1,\dl_2 \in \llun$ it holds 
 \begin{equation}\label{eq:bound2}
\|S'( \ell;\dl_1)-S'( \ell;\dl_2)\|_{L^2(0,T;\hon)} \leq K\, \|\dl_1-\dl_2\|_{L^2(0,T;\hoon)},\end{equation}where $S: \llun \to  \hhyy$ is the solution operator to \eqref{eq:q1}.
\end{lemma}\begin{proof}
In the proof of Proposition \ref{lem:ode1} we established that \eqref{eq:q1}  fits in the setting of section \ref{sec:evi} with the quantities from \eqref{eq:assu}. This means that $S: \llun \ni \ell \mapsto q \in   \hhyy$ is Lipschitz continuous according to Proposition \ref{prop:lip}. In view of Proposition \ref{lem:ode1_diff}, we can now conclude the desired estimate.
\end{proof}

The main result of this subsection reads as follows.
\begin{theorem}[Strong stationarity for the optimal control of the viscous damage model with fatigue]\label{thm:ss_qsep1}
 Let $\bar \ell \in \llz$ be locally optimal for \eqref{eq:optprob_q1} with associated state  $
  \bar q \in \hhyy.$   Then, there exists a unique adjoint state
$
  \xi \in H^{1}_T(0,T;\hoon)$ and a unique multiplier $\lambda \in L^{2}(0,T;H_0^1(\O))$ such that the following system is satisfied 
 \begin{subequations}\label{eq:strongstat_q1}
 \begin{gather}
  -\dot \xi-\alpha \Delta \lambda+[\F'(\bar q)]^\star \lambda= \partial_{q} \JJ(\bar q, \ll) \  \text{ in  }L^{2}(0,T;\hoon), \quad \xi(T) = 0,\label{eq:adjoint1_q1} \\[1mm]
\dual{\xi(t)}{  \frac{1}{\epsilon} (-\Delta)^{-1} (\mathbb{I}-P_{T(\bar z(t))})\,v}_{\hon}  \geq  \dual{\lambda(t)}{v}_{\hoon}   \quad \forall\, v\in \hoon, \;\text{a.e.\ in  }  \, (0,T),\label{eq:signcond_q1}
\\[1mm] \lambda+\partial_\ell \JJ(\bar q, \ll)=0\ \ \text{ in } L^2(0,T;\lo),\label{eq:grad_q1}
 \end{gather} 
 \end{subequations}    where we abbreviate $\bar z:=\alpha \Delta \bar q+\ll-\F(\bar q )$. Again, $$T(\bar z(t))=\overline{\R^+( \partial I_\CC(0 )-P_{\partial I_\CC(0 )}\bar z(t))}^{\hoon} \cap [ \dot { \bar q}(t)]^{\perp}$$ denotes {the critical cone of $\partial I_\CC(0 ) \subset \hoon$ at $(\bar z(t),\dot {\bar q}(t))$}.\end{theorem}

\begin{proof}
We aim to apply the strong stationarity result given by Theorem \ref{thm:strongstat} for the optimal control problem \eqref{eq:optprob_q1}. To this end, we have to check if \eqref{eq:optprob_q1} fits in the general setting from section \ref{sec:pres}, see Assumption \ref{assu:stand}. After that, we  verify Assumptions \ref{assu:s_diff}, \ref{assu:surj} and \ref{assu:reg}. 
Indeed, with the notations from section \ref{sec:pres}, we see that if we set
       \begin{subequations}\label{eq:assuu1}\begin{gather}
     V:=L^{2}(\Omega), \   Y:=H^{1}_0(\Omega),\  U:=L^{2}(\Omega), \  J:=\JJ, \\
                       f: Y^\ast \to Y, \quad f( \omega)= \frac{1}{\epsilon} (-\Delta)^{-1} (\mathbb{I}-P_{\partial I_\CC (0)})(\omega), \label{eq:as_f1} \\
            \Phi:L^2(0,T; Y \times U) \ni (q,u) \mapsto \alpha \Delta q +u-\F(q) \in L^2(0,T;Y^\ast), \label{eq:as_phi1}  \\ 
            \Psi : L^2(0,T; Y \times U) \ni(q,u)  \mapsto u \in L^2(0,T;  U^\ast ) \label{eq:as_psi1} , 
 \end{gather}
 \end{subequations}then \eqref{eq:optprob} coincides with \eqref{eq:optprob_q1}, thanks to Proposition \ref{lem:ode1}. Clearly, Assumption \ref{assu:stand}.\ref{it:stand1} is satisfied. Since $\F:\llz \to \llz$ is G\^ateaux-differentiable, see\,Assumption \ref{assu:standd}, the requirement in  Assumption \ref{assu:stand}.\ref{it:stand3} is fulfilled as well. Assumption \ref{assu:stand}.\ref{it:stand2} is satisfied by the mapping introduced in \eqref{eq:as_f1}, see Lemma \ref{f_dir}. Thus, the entire Assumption \ref{assu:stand} holds true in the setting \eqref{eq:assuu1}, cf.\ also Assumption \ref{assu:j_ex}.
 
 Moreover,  since for the setting considered in \eqref{eq:assuu1}, \eqref{eq:pde} is equivalent to \eqref{eq:syst_diff1}, we see that, in view of Proposition \ref{lem:ode1}, Assumption \ref{assu:s_diff}.\ref{it:s1} holds. The solution operator of \eqref{eq:pde} is given by $L^2(0,T;\lo) \ni \ell \mapsto (S(\ell),\ell) \in   \hhyy \times L^2(0,T;\lo)$, where $S: \llun \ni \ell \mapsto q \in   \hhyy$ is the solution operator of \eqref{eq:syst_diff1}. According to Proposition  \ref{lem:ode1_diff}, this is  directionally differentiable and its directional derivative $S'(\ell;\dl)$ at $\ell  $ in direction $\dl $ is the unique solution of \eqref{eq:syst_dif}. In light of Lemma \ref{f_dir}, this means that the pair $(\delta q,\delta u)=(S'(\ell;\dl),\dl)$ satisfies 
 \begin{equation*}\label{eq:syst_dif2}\begin{aligned}
 \dot {\delta q}(t)   = f'(\Phi(q,u)(t);\Phi'(q,u)(\delta q,\delta u)(t)) ,  \quad \delta q(0) = 0,
 \\\Psi'(q,u)(\delta q,\delta u)(t)=\delta \ell(t) \quad \text{\ae}(0,T).
\end{aligned} \end{equation*}that is, 
 \eqref{eq:pde_lin}. Thus, Assumption \ref{assu:s_diff}.\ref{it:s2} holds true.
Further, \eqref{eq:bound2} in combination with the embedding $\lo \embed \hoon$  implies that Assumption \ref{assu:s_diff}.\ref{it:s3} is verified as well; note that the second statement in Assumption \ref{assu:s_diff}.\ref{it:s3} is true in our setting, since $S'(\ell;\cdot):L^2(0,T;\lo) \to L^2(0,T;\hon)$ is continuous for any $\ell \in L^2(0,T;\lo)$.
Hence, the entire Assumption \ref{assu:s_diff} is true for the quantities in \eqref{eq:assuu}.

It remains to check that Assumptions \ref{assu:surj} and \ref{assu:reg} are guaranteed. To this end, we observe that 
  \begin{equation*}\label{eq:der_phi1}
\partial_{u}  \Phi (\bar q,\bar u)=\mathbb{I}:  \llz \to L^2(0,T;\hoon).
\end{equation*}
As a result of $\llz \dense L^2(0,T;\hoon)$, the 'constraint qualification' in Assumption \ref{assu:surj} is fulfilled. In the light of  \eqref{eq:as_phi1}-\eqref{eq:as_psi1}, the adjoints of the partial derivatives of $\Phi$ and $\Psi$ are given by  
\begin{equation}\label{eq:addj1}
\begin{aligned}
 \partial_{q}  \Phi (\bar q,\bar u)^\ast &= \alpha \Delta -[\F'(\bar q)]^\star: L^2(0,T;\hon)   \to   L^2(0,T;\hoon),\\
\partial_{u}  \Phi (\bar q,\bar u)^\ast &= \mathbb{I}: L^2(0,T;\hon)  \to   L^2(0,T;\lo),\\
\partial_{q}  \Psi (\bar q,\bar u)^\ast &= 0: L^2(0,T;\lo)  \to   L^2(0,T;\hoon),\\
\partial_{u}  \Psi (\bar q,\bar u)^\ast &= \mathbb{I}: L^2(0,T;\lo)   \to L^2(0,T;\lo).
\end{aligned}
\end{equation}
To see  that Assumption \ref{assu:reg} is true, we only need to check if $ \partial_\ell \JJ(\bar q, \ll)  \in  L^2(0,T;\hon)$ (cf.\,\eqref{eq:addj1} and note that $\partial_u \JJ=0$). To this end, we make use of \eqref{eq:vi}, which in the setting \eqref{eq:assuu1} reads
$$ \partial_{q} \JJ(\bar q,\ll) S'(\bar \ell; \delta \ell) + \partial_{\ell} \JJ(\bar q,\ll) \delta \ell \geq 0  \quad \forall\, \delta \ell \in L^2(0,T;L^2(\O)).$$
Since $$
\|S'( \ll;\dl)\|_{L^2(0,T;\hon)} \leq K\, \|\dl\|_{L^2(0,T;\hoon)} \quad \forall\,\dl \in L^2(0,T;\hoon),$$see \eqref{eq:bound2}, Hahn-Banach theorem now gives in turn that $ \partial_\ell \JJ(\bar q, \ll)  \in  L^2(0,T;\hon)$.

Thus, we can apply Theorem \ref{thm:strongstat}, which  in combination with \eqref{eq:addj1} tells us that there exist unique adjoint "states" $\xi \in H^1_T(0,T;\hoon)$ and $w \in L^2(0,T;\lo)$ and a unique multiplier $\lambda \in L^2(0,T;\hon)$ such that 
 \begin{subequations}\label{eq:strongstat_qq1}
 \begin{gather}
-\dot \xi-\alpha \Delta \lambda+[\F'(\bar q)]^\star \lambda= \partial_{q} \JJ(\bar q, \ll) \  \text{ in  }L^{2}(0,T;\hoon), \quad \xi(T) = 0,\label{eq:adjoint1_qq1} \\[1mm]
- \lambda+ w= 0 \  \text{ in  }L^{2}(0,T;\lo), \label{eq:adjoint1_qq2} \\[1mm]
\dual{\xi(t)}{ f' (\Phi (\bar q , \bar u)(t) ; v)}_{\hon}  \geq  \dual{\lambda(t)}{v}_{\hoon}   \quad \forall\, v\in \hoon, \;\text{a.e.\ in  }  \, (0,T), \label{eq:signcond_qq1} \\[1mm]
w+\partial_\ell \JJ(\bar q, \ll)=0\ \ \text{ in } L^2(0,T;L^2(\Omega)). \label{eq:grad_qq1}
 \end{gather} 
 \end{subequations}  
Inserting \eqref{eq:adjoint1_qq2} in \eqref{eq:grad_qq1} and employing Lemma \ref{f_dir} to compute the derivative in \eqref{eq:signcond_qq1} finally yields \eqref{eq:strongstat_q1}; recall that $\psi(z(t))=\dot q (t)$, see the end of the proof of Proposition \ref{lem:ode1_diff} and Lemma \ref{f_dir}. The proof is now complete.
\end{proof}

The optimality system in Theorem \ref{thm:ss_qsep1} is indeed of strong stationary type, as the next result shows:

\begin{theorem}[Equivalence between B- and strong stationarity]\label{thm:B_qsep1} Assume that $\bar \ell \in L^2(0,T;L^2(\O))$ together with its state $\bar q \in \hhyy $, some 
 adjoint state $\xi \in H^{1}_T(0,T;\hoon)$ and a multiplier $\lambda \in L^{2}(0,T;H_0^1(\O))$ 
 satisfy the optimality system \eqref{eq:strongstat_q1}.
 Then, it also satisfies the variational inequality 
  \begin{equation}\label{eq:vi_max_q1}
  \partial_{q} \JJ(\bar q,\ll) S'(\bar \ell; \delta \ell) + \partial_{\ell} \JJ(\bar q,\ll) \delta \ell \geq 0  \quad \forall\, \delta \ell \in L^2(0,T;L^2(\O)),
 \end{equation}where $S$ is the solution mapping associated to \eqref{eq:q1}, see Proposition \ref{lem:ode1}.
 \end{theorem}

\begin{proof}We show the result by means of Theorem \ref{thm:equiv_B_strong}. 
In the proof of Theorem \ref{thm:ss_qsep1}, we have seen that the problem \eqref{eq:optprob_q1} fits in the setting from section \ref{sec:pres}, i.e., Assumptions \ref{assu:stand} and \ref{assu:s_diff} are satisfied for the quantities in \eqref{eq:assuu1}. According to the end of the  proof of Theorem \ref{thm:ss_qsep1}, the system \eqref{eq:strongstat_q1} coincides with \eqref{eq:strongstat_qq1}, which is the same as \eqref{eq:strongstat}  in this particular setting, see \eqref{eq:addj1}. 
We also note that \eqref{eq:vi} is just \eqref{eq:vi_max_q1}. Hence, the desired statement is true.
Note that, since Assumptions \ref{assu:surj} and \ref{assu:reg} are fulfilled, cf.\ the proof of Theorem \ref{thm:ss_qsep1}, we have the equivalence $\eqref{eq:vi_max_q1} \Longleftrightarrow \eqref{eq:strongstat_q1}$.
\end{proof}

\subsection{Penalization ($L^2-$viscosity)}\label{sec:1}

In this subsection we apply the result from section \ref{sec:pres} to obtain strong stationary optimality conditions for the  control of a \textit{two-field} gradient damage model with fatigue. 
The problem we consider is a penalized version of the viscous fatigue damage model addressed in \cite{alessi}. This kind of penalization has already been proven to be successful  in the context of classical damage models (without fatigue). 
Firstly, it  approximates the single-field damage model, cf.\ \cite{paper2}, and secondly, it is frequently employed in computational mechanics (see e.g.\ \cite{hackl} and the references therein). 

The model we intend to examine describes the evolution of damage under the influence of a time-dependent load $\ell:[0,T] \to H^1(\Omega)^\ast$ (control) acting on a body occupying the bounded Lipschitz domain $\O \subset \R^N$, $N \in \{2,3\}$. The induced 'local' and 'nonlocal' damage are expressed in terms of $\varphi : [0,T] \to H^{1}(\Omega)$ and $d : [0,T] \to L^2(\Omega)$, respectively (states). For more details, see \cite[Sec.\ 1]{hal}.

The viscous two-field gradient damage problem with fatigue reads as follows:
\begin{equation}\label{eq:q}
\left.
  \begin{gathered}
   \varphi(t) \in \argmin_{\varphi \in H^1(\Omega)} \EE(t,  \varphi, d(t)),
    \\ -\partial_d \EE(t,\varphi(t),d(t)) \in \partial_2 \RR_\epsilon (H(d)(t),\dot{d}(t)) \quad \text{in }\lo , \quad d(0) = 0
  \end{gathered}\ \right\}
  \end{equation}a.e.\ in $(0,T)$.    In \eqref{eq:q}, the stored energy $\EE:[0,T] \times H^{1}(\Omega) \times L^2(\Omega)\rightarrow \mathbb{R}$ is given by
\begin{equation*}\label{eq:e}\begin{aligned}
\EE(t, \varphi,d)&:=\frac{\alpha}{2}\|\nabla \varphi\|_{L^2(\O)}^2+\frac{\beta}{2}\| \varphi - d\|_{L^2(\O)}^2-\dual{\ell(t)}{\varphi}_{H^{1}(\Omega)},
\end{aligned}\end{equation*}where $\alpha>0$ is the gradient regularization and $ \beta >0$ denotes the {penalization parameter}. 
   The viscous dissipation  $\RR_\epsilon:L^2(\O) \times  L^2(\O) \rightarrow (-\infty,\infty]$ is defined as \begin{equation*}\label{def:r}
\RR_\epsilon(\zeta,\eta):=\left\{ \begin{aligned}\int_\Omega \kappa(\zeta) \,  \eta \;dx +\frac{\epsilon}{2}\|\eta\|^2_{\lo}, &\quad \text{if }\eta \geq 0 \text{ a.e. in }\Omega,
\\\infty  &\quad \text{otherwise,}\end{aligned} \right.
\end{equation*}
where  $\epsilon>0$ is the viscosity parameter.

\begin{remark}
Optimality conditions for the control of \eqref{eq:q} have been recently established in \cite{hal}, however in a more general setting. Therein, the fatigue degradation function $\kappa$  is assumed to be only directionally differentiable, so that conditions of strong stationary type are not to be expected \cite[Remark 3.22]{hal}. Moreover, in the contribution \cite{hal}, the control space is $H^1(0,T;\lo)$, whereas we will work with $\llz$. Note that there is no need to apply the findings in section \ref{sec:evi} to see that \eqref{eq:q} is a system of the type \eqref{eq:pde}; this is already stated in \cite[Prop.\,2.3]{st_coup}, which is stated below.
\end{remark}

\begin{lemma}{\cite[Propositions 2.3, 2.6]{hal}}[Control-to-state map  and directional differentiability]\label{lem:ode}
 \begin{enumerate} \item\label{it:l1}For every $\ell \in \llU$, the viscous damage problem with fatigue \eqref{eq:q} admits a unique solution $(d,\varphi) \in   \hhyn \times \llu $, which is characterized by  the following PDE system
\begin{subequations}\label{eq:syst_diff}
\begin{gather}
\dot d(t)   =   \frac{1}{\epsilon} \max (-\beta(d(t)-\varphi(t))-(\kappa \circ H)(d)(t), 0)\ \text{ in }\lo, \quad d(0) = 0,    \label{eq:syst_diff11} 
\\- \alpha \Delta \varphi(t) + \beta \,\varphi(t)   
= \beta d(t) +\ell(t) \quad  \text{in }H^{1}(\Omega)^\ast \label{eq:syst_diff12} 
\end{gather}
\end{subequations}a.e.\ in $(0,T)$. 
\item\label{it:l2}
The solution map associated to \eqref{eq:q} $$S: \llU \ni \ell \mapsto (d,\varphi) \in   \hhyn \times \llu$$ is directionally differentiable. 
Its directional derivative $(\dd,\df):=S'(\ell;\dl)$ at $\ell  \in \llU$ in direction $\dl \in \llU$ is the unique solution of
 \begin{equation}\label{eq:ode_lin_q}\begin{aligned}
\dot \dd(t)&=  \frac{1}{\epsilon} \maxlim ' \big(z(t); -\beta(\dd(t)-\df(t))-\F'(d)(\delta d)(t)\big)\ \text{ in }\lo, \ \, \dd(0)=0, \\
&- \alpha \Delta \df(t) + \beta \,\df(t) = \beta \dd(t) +\dl(t) \quad  \text{in }H^{1}(\Omega)^\ast
\end{aligned}
 \end{equation}a.e.\ in $(0,T)$, where we abbreviate $ z(t):=-\beta(d(t)-\varphi(t))-(\kappa \circ H)(d)(t)$.  
 \end{enumerate}      \end{lemma}

Next, we want to apply the strong stationarity result from section \ref{sec:pres} to the following
 optimal control problem:

\begin{equation}\tag{Q}\label{eq:optprob_q}
 \left.
 \begin{aligned}
  \min_{\ell \in \llz} \quad & \JJ(d,\varphi,\ell)
  \\
 \text{s.t.} \quad & (d,\varphi) \text{ solves }\eqref{eq:q} \text{ with r.h.s.\ }\ell.
 \end{aligned}
 \quad \right\}
\end{equation}

In the sequel, the objective $\JJ$ is supposed to fulfill 
\begin{assumption}\label{assu:j_ex}
  The objective functional $\JJ: L^2(0,T;\lo)\times \llu \times \llz \to \R$ 
  is continuously Fr\'echet-differentiable and  Lipschitz continuous on bounded sets, i.e.,
for all $M>0$ there exists $L_M>0$ so that 
  \begin{equation*}\label{J_lip}
  |\JJ(v_1)-\JJ(v_2)| \leq L_M\,\|v_1-v_2\|_X  \quad \forall\, v_1,v_2 \in B_X(0,M),
  \end{equation*}where we abbreviate $X:=L^2(0,T;\lo)\times \llu \times \llz$.
\end{assumption}

Note that Assumption \ref{assu:j_ex} is satisfied by classical objectives of tracking type such as
\begin{equation*}
\JJ_{ex}( d ,\varphi, \ell) := \frac{1}{2}\, \|\varphi-\varphi_d\|^2_{L^2(0,T;H^1(\Omega))}
 + \frac{\alpha_1}{2} \|d\|^2_{L^2(0,T;L^2(\Omega))}+\frac{\alpha_2}{2} \|\ell-\ell_d\|^2_{L^2(0,T;\lo)},
\end{equation*}
where $\varphi_d \in L^2(0,T;H^1(\Omega))$, $\ell_d \in L^2(0,T;\lo)$ and $\alpha_1,\alpha_2 > 0$.

Before stating the strong stationary optimality  conditions, we check that Assumption \ref{assu:reg}  is satisfied in our setting. As it will turn out in the proof of Theorem \ref{thm:ss_qsep} below, this is indeed the case, as a result of the following

\begin{lemma}\label{lem:reg}
For any local optimum $\bar \ell$ of \eqref{eq:optprob_q}, there exists $(\lambda,w)\in \llz \times L^2(0,T;\ho)$ so that 
\begin{subequations} \begin{gather}
-\alpha  \Delta w(t)  + \beta (w(t)-\lambda(t) ) =\partial_\varphi \JJ(S(\ll), \ll)(t) \ \ \text{ in } H^{1}(\Omega)^\ast, \label{eq:adj_syst2p}\\
w(t)+\partial_\ell \JJ(S(\ll), \ll)(t)=0\ \ \text{ in } H^{1}(\Omega), \quad \text{a.e.\ in }(0,T),\label{eq:gradgl}
\end{gather}\end{subequations} 
where $S$ is the solution operator associated to \eqref{eq:q}, see Lemma \ref{lem:ode}.\ref{it:l2}.
\end{lemma}
\begin{proof}
We refer to the proof of \cite[Lem.\,4.3]{st_coup}, where a very similar setting is considered. The result is shown by a classical regularization approach, see \cite{tiba, Barbu:1981:NCD} for instance. One defines a smooth approximation of the function ${\max(\cdot, 0)}$,  to which one associates a state equation where the solution mapping is G\^{a}teaux-differentiable. Then, it is shown that $\ll$ can be approximated by a sequence of local minimizers of an optimal control problem governed by the regularized state equation. Passing to the limit in the adjoint system associated to the regularized optimal control problem finally yields the desired assertion. The only difference to the proof of \cite[Lem.\,4.3]{st_coup} consists in having to show the existence and uniformly boundedness of the regularized adjoint state. To this end, one can follow  arguments employed in the proof of \cite[Prop.\,3.10]{hal}.
\end{proof}

\begin{remark}\label{rem:lambda}
Note that the existence of a tuple $(\xi,\lambda,w)\in H^1_T(0,T;\hoo) \times L^2(0,T;\hoo) \times L^2(0,T;\ho)$ satisfying the adjoint system \eqref{eq:adjoint1_q}-\eqref{eq:adjoint2_q}-\eqref{eq:grad_q} below  follows directly, without employing Lemma \ref{lem:reg}. It is the $\llz$-regularity of the multiplier $\lambda$ which cannot be immediately deduced from \eqref{eq:adjoint2_q}. The fact that this type of additional information may follow from a regularization approach is not new, see   \cite{paper, st_coup} for similar situations. Let us underline that the improved space-regularity of $\lambda$ is essential for deriving the sign condition in \eqref{eq:signcond_q} a.e.\ in $(0,T) \times \O$, see the proof of Theorem \ref{thm:ss_qsep} below.
\end{remark}

The main result of this subsection reads as follows.
\begin{theorem}[Strong stationarity for the optimal control of the viscous two-field  damage model with fatigue]\label{thm:ss_qsep}
 Let $\bar \ell \in \llz$ be locally optimal for \eqref{eq:optprob_q} with associated states  \begin{equation*}
  \bar d \in \hhyn \quad \text{and} \quad \bar \varphi \in \llu. 
 \end{equation*}   Then, there exist unique adjoint states
 \begin{equation*}
  \xi \in H^{1}_T(0,T;\lo) \quad \text{and} \quad w \in \llu,
 \end{equation*}  
 and a unique multiplier $\lambda \in L^{2}(0,T;\lo)$ such that the following system is satisfied 
 \vspace{-0.5cm}
 \begin{subequations}\label{eq:strongstat_q}
 \begin{gather}
  -\dot \xi- \beta (w-\lambda)+[\F'(\bar d)]^\ast \lambda= \partial_{d} \JJ(\bar d,\bar \varphi, \ll) \  \text{ in  }L^{2}(0,T;\lo), \quad \xi(T) = 0,\label{eq:adjoint1_q} \\[1mm]
-\alpha  \Delta w + \beta (w-\lambda)=\partial_\varphi \JJ(\bar d,\bar \varphi, \ll) \ \ \text{ in } L^2(0,T;H^{1}(\Omega)^\ast), \label{eq:adjoint2_q}\\[1mm]
 \left. \begin{aligned} 
  \lambda(t,x)&=\frac{1}{\epsilon} \raisebox{3pt}{$\chi$}_{\{ \bar  z>0\}}(t,x) \xi(t,x) \quad \text{a.e.\ where  }\bar z(t,x) \neq 0,\\
  \lambda(t,x) &\in \big[0,\frac{1}{\epsilon} \xi(t,x) \big]  \quad \text{a.e.\ where  }\bar z(t,x) = 0,\end{aligned}\right\}\label{eq:signcond_q}
\\[1mm] w+\partial_\ell \JJ(\bar d,\bar \varphi, \ll)=0\ \ \text{ in } L^2(0,T;H^{1}(\Omega)),\label{eq:grad_q}
 \end{gather} 
 \end{subequations}    where we abbreviate $\bar z:=-\beta(\bar d -\bar \varphi)-\F(\bar d )$. 
\end{theorem}

\begin{proof}We aim to apply the strong stationarity result given by Theorem \ref{thm:strongstat} for the optimal control problem \eqref{eq:optprob_q}. To this end, we have to check if \eqref{eq:optprob_q} fits in the general setting from section \ref{sec:pres}. After that, we  verify Assumptions \ref{assu:s_diff}, \ref{assu:surj} and \ref{assu:reg}. 
Indeed, with the notations from section \ref{sec:pres}, we see that if we set
       \begin{subequations}\label{eq:assuu}\begin{gather}
     V:=L^{2}(\Omega), \   Y:=\lo,\  U:=H^{1}(\Omega), \  J:=\JJ, \\
                       f: Y^\ast \to Y, \quad f( \omega)= \frac{1}{\epsilon}  \max(\omega , {0}), \label{eq:as_f} \\
            \Phi:L^2(0,T; Y \times U) \ni (d,\varphi) \mapsto -\beta( d - \varphi)-\F(d) \in L^2(0,T;Y^\ast), \label{eq:as_phi}  \\ 
            \Psi : L^2(0,T; Y \times U) \ni(d,\varphi)  \mapsto -\alpha  \Delta \varphi + \beta \varphi - \beta d \in L^2(0,T;  U^\ast ) \label{eq:as_psi} , 
 \end{gather}
 \end{subequations}then \eqref{eq:optprob} coincides with \eqref{eq:optprob_q}, thanks to Lemma \ref{lem:ode}.\ref{it:l1}. Notice that $V \dense U^\ast $ so that Assumption \ref{assu:stand}.\ref{it:stand1} is satisfied. Since $\F:\llz \to \llz$ is G\^ateaux-differentiable, see\,Assumption \ref{assu:standd}, the requirement in  Assumption \ref{assu:stand}.\ref{it:stand3} is fulfilled as well. Since $\max(\cdot,0):\lo \to \lo$ is Lipschitz continuous and directionally differentiable, Assumption \ref{assu:stand}.\ref{it:stand2} is satisfied by \eqref{eq:as_f}. Thus, the entire Assumption \ref{assu:stand} holds true in the setting \eqref{eq:assuu}, cf.\ also Assumption \ref{assu:j_ex}.

 Moreover, by employing  Lemma \ref{lem:ode}.\ref{it:l1}, we see that Assumption \ref{assu:s_diff}.\ref{it:s1} holds. The resulting solution operator of \eqref{eq:pde}, i.e., $$S: \llU \ni \ell \mapsto (d,\varphi) \in   \hhyn \times \llu$$ is directionally differentiable, cf.\ Lemma  \ref{lem:ode}.\ref{it:l2}. According to the latter, its directional derivative $S'(\ell;\dl)$ at $\ell  $ in direction $\dl $ is the unique solution of \eqref{eq:ode_lin_q}, and thus, of 
 \eqref{eq:pde_lin}, whence Assumption \ref{assu:s_diff}.\ref{it:s2} follows. 
According to \cite[Lemma 2.4]{hal}, $S: \llU \to   \hhyn \times \llu$ is Lipschitz continuous, which implies that Assumption \ref{assu:s_diff}.\ref{it:s3} is verified as well. Hence, the entire Assumption \ref{assu:s_diff} is true for the setting \eqref{eq:assuu}.
 

It remains to check that Assumptions \ref{assu:surj} and \ref{assu:reg} are guaranteed. To this end, we observe that 
  \begin{equation*}\label{eq:der_phi}
\partial_{\varphi}  \Phi (\bar d,\bar \varphi)=\beta \,\mathbb{I}:  \llu \to L^2(0,T;\lo).
\end{equation*}
As a result of $\llu \dense L^2(0,T;\lo)$, the 'constraint qualification' in Assumption \ref{assu:surj} is fulfilled. In the light of  \eqref{eq:as_phi}-\eqref{eq:as_psi}, the adjoints of the partial derivatives of $\Phi$ and $\Psi$ are given by  
\begin{equation}\label{eq:addj}
\begin{aligned}
 \partial_{d}  \Phi (\bar d,\bar \varphi)^\ast &= -\beta \,\mathbb{I}-[\F'(\bar d)]^\ast : L^2(0,T;\lo)   \to   L^2(0,T;\lo),\\
\partial_{\varphi}  \Phi (\bar d,\bar \varphi)^\ast &= \beta \,\mathbb{I}: L^2(0,T;\lo)  \to   L^2(0,T;\hoo),\\
\partial_{d}  \Psi (\bar d,\bar \varphi)^\ast &= -\beta \,\mathbb{I}: \llu   \to   L^2(0,T;\lo),\\
\partial_{\varphi}  \Psi (\bar d,\bar \varphi)^\ast &= -\alpha \Delta + \beta \,\mathbb{I}: \llu   \to L^2(0,T;\hoo).
\end{aligned}
\end{equation}
Now Lemma \ref{lem:reg} gives in turn that Assumption \ref{assu:reg} is true for the setting \eqref{eq:assuu}.
Thus, we can apply Theorem \ref{thm:strongstat}, which  in combination with \eqref{eq:addj} tells us that there exist unique adjoint states $\xi \in H^1_T(0,T;\lo)$ and $w \in \llu$ and a unique multiplier $\lambda \in L^2(0,T;\lo)$ such that 
 \begin{subequations}\label{eq:strongstat_qq}
 \begin{gather}
-\dot \xi+\beta \lambda+[\F'(\bar d)]^\ast \lambda- \beta w= \partial_{d} \JJ(\bar d,\bar \varphi, \ll) \  \text{ in  }L^{2}(0,T;\lo), \quad \xi(T) = 0,\label{eq:adjoint1_qq} \\[1mm]
  -\beta \lambda-\alpha  \Delta w + \beta w=\partial_\varphi \JJ(\bar d,\bar \varphi, \ll) \ \ \text{ in } L^2(0,T;H^{1}(\Omega)^\ast), \label{eq:adjoint2_qq}\\[1mm]
\big(\xi(t), f' (\Phi (\bar d , \bar \varphi)(t) ; v)\big)_{\lo}  \geq  (\lambda(t),v)_{\lo}   \quad \forall\, v\in \lo, \;\text{a.e.\ in  }  \, (0,T), \label{eq:signcond_qq} \\[1mm]
w+\partial_\ell \JJ(\bar d,\bar \varphi, \ll)=0\ \ \text{ in } L^2(0,T;H^{1}(\Omega)). \label{eq:grad_qq}
 \end{gather} 
 \end{subequations}  
It remains to show that \eqref{eq:signcond_qq} implies \eqref{eq:signcond_q}. Here, we recall the abbreviation $\bar z:=-\beta(\bar d-\bar \varphi)-\F(\bar d)$ and \eqref{eq:as_phi}, i.e., $\bar z=\Phi (\bar  d,\bar \varphi)$. An argument based on the fundamental lemma of calculus of variations and the positive homogeneity of the directional derivative w.r.t.\ direction yields
    \begin{equation*}\label{eq:d1}
\frac{1}{\epsilon} \xi(t,x) \maxlim ' (\bar z(t,x);1)  \geq  \lambda(t,x) \geq -\frac{1}{\epsilon} \xi(t,x) \maxlim ' (\bar z(t,x);-1)
  \quad  \;\text{a.e.\ in  }  \, (0,T) \times \O,
   \end{equation*} in view of \eqref{eq:as_f}.
The desired assertion now follows by distinguishing between the sets $\{(t,x):\bar z (t,x) >0\}$, $\{(t,x):\bar z (t,x) <0\}$ and $\{(t,x):\bar z (t,x) =0\}$.
\end{proof}

\begin{remark}
 If $\bar z(t,x) \neq  0$ \ae $(0,T) \times \O$, then  \eqref{eq:strongstat_q} reduces to the standard KKT-conditions, since in this case, \eqref{eq:signcond_q} is equivalent to 
 \[\lambda=\frac{1}{\epsilon}\maxlim '(\bar z)\xi \quad \ae (0,T) \times \O.\]
\end{remark}

The optimality system in Theorem \ref{thm:ss_qsep} is indeed of strong stationary type, as the next result shows:

\begin{theorem}[Equivalence between B- and strong stationarity]\label{thm:B_qsep} Assume that $\bar \ell \in L^2(0,T;L^2(\O))$ together with its states $(\bar d, \bar \varphi) \in \hhyn \times \llu$, some 
 adjoint states $(\xi,w) \in H^{1}_T(0,T;\lo) \times \llu$, and a multiplier $\lambda \in L^{2}(0,T;\lo)$ 
 satisfy the optimality system \eqref{eq:adjoint1_q}--\eqref{eq:grad_q}.
 Then, it also satisfies the variational inequality 
  \begin{equation}\label{eq:vi_max_q}
  \partial_{( d,  \varphi)} \JJ(\bar d, \bar \varphi,\ll) S'(\bar \ell; \delta \ell) + \partial_{\ell} \JJ(\bar d, \bar \varphi,\ll) \delta \ell \geq 0  \quad \forall\, \delta \ell \in L^2(0,T;L^2(\O)),
 \end{equation}where $S: \llU  \to  \hhyn \times \llu$ is the solution mapping associated to \eqref{eq:q}, see Lemma \ref{lem:ode}.
 \end{theorem}

\begin{proof}We show the result by means of Theorem \ref{thm:equiv_B_strong}. 
In the proof of Theorem \ref{thm:ss_qsep}, we have seen that the problem \eqref{eq:optprob_q} fits in the setting from Section \ref{sec:pres}, i.e., Assumptions \ref{assu:stand} and \ref{assu:s_diff} are satisfied for the quantities in \eqref{eq:assuu}. According to the proof of Theorem \ref{thm:ss_qsep}, the system \eqref{eq:strongstat} coincides with \eqref{eq:strongstat_qq} in this particular setting, see \eqref{eq:addj}. 
We also note that \eqref{eq:vi} is just \eqref{eq:vi_max_q}. 
Thus, in view of Theorem \ref{thm:equiv_B_strong}, we only need to show that \eqref{eq:signcond_q} implies \eqref{eq:signcond_qq}, which, in view of \eqref{eq:as_f} and \eqref{eq:as_phi}, reads
\begin{equation}\label{eq:to_show}
\big(\xi(t),\frac{1}{\epsilon} \maxlim ' (\bar z(t); v)\big)_{\lo}  \geq  (\lambda(t),v)_{\lo}   \quad \forall\, v\in \lo, \;\text{a.e.\ in  }  \, (0,T),  \end{equation}where $\bar z:=-\beta(\bar d -\bar \varphi)-\F(\bar d)$. 

  To this end, let $v\in \lo$ be arbitrary, but fixed. 
  From the first identity in \eqref{eq:signcond_q}, we know that  \begin{equation} \label{eq:da} \begin{aligned}
\lambda(t,x)v(x) &=\frac{1}{\epsilon} \raisebox{3pt}{$\chi$}_{\{ \bar  z>0\}}(t,x)v(x) \xi(t,x)
\\&=\frac{1}{\epsilon} \maxlim '(\bar z(t,x))v(x)\xi(t,x) \quad \text{a.e.\ where } \bar z(t,x) \neq 0.
\end{aligned} \end{equation}

Further, we define $M^+:=\{(t,x) \in (0,T) \times  \O: \bar z(t,x) = 0 \text{ and } v(x)>0 \}$ and $M^-:=\{(t,x) \in (0,T) \times  \O: \bar z(t,x) = 0 \text{ and }  v(x) \leq 0 \}$ (up to sets of measure zero).   Then, the second identity in \eqref{eq:signcond_q} yields  
\begin{equation}  \label{eq:db} \begin{aligned}
 \lambda(t,x)v(x)  &\leq \left \{
 \begin{aligned}
 &\frac{1}{\epsilon}  \xi(t,x)v(x)
&& \ \ \text{ a.e.\ in  }M^+
\\ &0 && \ \  \text{ a.e.\ in  }M^-
\end{aligned} \right.
\\&\quad =\frac{1}{\epsilon} \maxlim '(\bar z(t,x);v(x))\xi(t,x) \quad \text{a.e.\ where } \bar z(t,x) = 0.
\end{aligned}
\end{equation}  Now, \eqref{eq:to_show} follows from \eqref{eq:da} and \eqref{eq:db}.
Note that, since Assumptions \ref{assu:surj} and \ref{assu:reg} are fulfilled, cf.\ the proof of Theorem \ref{thm:ss_qsep}, we have the equivalence $\eqref{eq:vi_max_q} \Longleftrightarrow \eqref{eq:strongstat_q}$.
\end{proof}

\section*{Acknowledgment}This work was supported by the DFG grant BE 7178/3-1 for the project "Optimal Control of Viscous
Fatigue Damage Models for Brittle Materials: Optimality Systems".

\bibliographystyle{plain}
\bibliography{hist_dep_EVI}

\end{document}